\documentclass[reqno]{amsart}
\usepackage{amssymb, amsthm, amsmath, amsfonts}
\usepackage{graphics}
\usepackage[all]{xy}
\usepackage{hyperref}
\usepackage{enumerate}
\usepackage[mathscr]{eucal}
\usepackage{thmtools}

\theoremstyle{plain}
\newtheorem{theorem}{Theorem}[section]
\newtheorem{lemma}[theorem]{Lemma}
\newtheorem{proposition}[theorem]{Proposition}
\newtheorem{corollary}[theorem]{Corollary}
\numberwithin{equation}{section}

\theoremstyle{definition}

\newtheorem{definition}[theorem]{Definition}
\newtheorem{example}[theorem]{Example}
\newtheorem{remark}[theorem]{Remark}

\newcommand{\C}{{\mathscr{C}}}
\newcommand{\FI}{{\mathrm{FI}}}
\newcommand{\VI}{{\mathrm{VI}}}
\newcommand{\OI}{{\mathrm{OI}}}
\newcommand{\CMod}{{\C\mbox{-Mod}}}
\newcommand{\N}{{\mathbb{N}}}
\newcommand{\CN}{{\mathrm{N}}}
\newcommand{\mm}{{\mathfrak{m}}}
\newcommand{\bfn}{{\mathbf{n}}}

\DeclareMathOperator{\Ob}{Ob}
\DeclareMathOperator{\gd}{gd}
\DeclareMathOperator{\hd}{hd}
\DeclareMathOperator{\reg}{reg}
\DeclareMathOperator{\Tor}{Tor}

\title[Asymptotic behavior of representations of graded categories]{Asymptotic behavior of representations of graded categories with inductive functors}

\author{Wee Liang Gan}
\address{Department of Mathematics, University of California, Riverside, CA 92521, USA}
\email{wlgan@math.ucr.edu}

\author{Liping Li}
\address{HPCSIP (Ministry of Education), College of Mathematics and Computer Science, Hunan Normal University, Changsha, Hunan 410081, China.}
\email{lipingli@hunnu.edu.cn}

\thanks{The second author is supported by the National Natural Science Foundation of China 11771135 and the Start-Up Funds of Hunan Normal University 830122-0037.}

\begin{document}

\begin{abstract}
In this paper we describe inductive machinery to investigate asymptotic behavior of homology groups and related invariants of representations of certain graded combinatorial categories over a commutative Noetherian ring $k$, via introducing inductive functors which generalize important properties of shift functors of $\FI$-modules. In particular, a sufficient criterion for finiteness of Castelnuovo-Mumford regularity of finitely generated representations of these categories is obtained. As applications, we show that a few important infinite combinatorial categories appearing in representation stability theory (for example $\FI^d$, $\OI^d$, $\FI_d$, $\OI_d$) are equipped with inductive functors, and hence the finiteness of Castelnuovo-Mumford regularity of their finitely generated representations is guaranteed. We also prove that truncated representations of these categories have linear minimal resolutions by relative projective modules, which are precisely linear minimal projective resolutions when $k$ is a field of characteristic 0.
\end{abstract}

\maketitle

\section{Introduction}

\subsection{Motivation}

Asymptotic behavior of homology groups and related invariants (such as syzygies, Betti numbers, Castelnuovo-Mumford regularity, etc.) of ideals and modules of graded algebras has been extensively considered in commutative algebra and representation theory; see for instance \cite{Av, AIS, AP, CD, CHT, DHS, EG, ES, HT, Hu, IR, Tr}. Recently, Church and Farb introduced in \cite{CF} the notion of \emph{representation stability}, providing a new framework to study asymptotic behavior of homology or cohomology groups of a family of topological spaces with certain symmetry via viewing these (co)homology groups as representations of different symmetry groups. Moreover, they and Ellenberg introduced in \cite{CEF} \emph{$\FI$-modules} to categorify representation stability theory. Soon after, quite a few infinite combinatorial categories were introduced to study asymptotic behavior of various modules and their homology groups. See for instance the foundational papers of Church, Ellenberg, Farb, Nagpal (eg. \cite{CE, CEF, CEFN}), Putman, Sam, Snowden (eg. \cite{PS, SS, SS3}), and Wilson (eg. \cite{Wi1, Wi2}). For applications in topology, geometric group theory, number theory or representation theory,  see for example the papers of Casto (eg. \cite{Ca1, Ca2}), Church-Ellenberg-Farb \cite{CEF2}, Church-Putman \cite{CP}, Ellenberg-Wiltshire-Gordon \cite{EWG}, Farb-Wolfson \cite{FW}, Gadish \cite{Gad2}, Harman \cite{Har1, Har2}, Jim\'enez Rolland, Wilson (eg. \cite{JR,  JRW1, JRW2}), Miller-Wilson (eg. \cite{MW1, MW2}), Miller-Patzt-Wilson \cite{MPW}, Nagpal  \cite{Nag}, Patzt \cite{Pa1}, Patzt-Wu \cite{PW}, Putman, Sam, and Snowden (eg. \cite{Put, PS, PSS, SS4}). Additionally, algebraic aspects of the theory have been studied, for example, by Djament (eg. \cite{Dj}), Gadish \cite{Gad}, Gan-Li (eg. \cite{GL5}-\cite{GL3}), Gan-Watterlond \cite{GW1, GW2}, Li (eg. \cite{L1, L2}), Li-Ramos \cite{LR}, Li-Yu \cite{LY1, LY2}, Nagpal \cite{Nag2}, Patzt \cite{Pa2}, Ramos \cite{R1, R2}, Sam and Snowden (eg. \cite{SS5, SS1}).

The following table contains a few interesting categories frequently appearing in literature.
\begin{center}
Table 1: A few infinite combinatorial categories
\begin{tabular}{c|cc}
  Categories & Objects & Morphisms\\
  \hline
  $\CN$ & $[n], n \in \N$ & the natural inclusions\\
  $\CN^d$ & $[n_1] \times \ldots \times [n_d], n_i \in \N$ & $d$-tuples of natural inclusions\\
  $\FI$ & $[n], n \in \N$ & injections\\
  $\FI_G$ & $[n], n \in \N$ & $(f, g): [n] \times [n] \to [m] \times G$ where $f$ is injective\\
  $\OI$ & $[n], n \in \N$ & order-preserving injections\\
  $\FI^d$ & $[n_1] \times \ldots \times [n_d], n_i \in \N$ & $d$-tuples of injections\\
  $\OI^d$ & $[n_1] \times \ldots \times [n_d], n_i \in \N$ & $d$-tuples of order-preserving injections\\
  $\FI_d$ & $[n], n \in \N$ & pairs of injections and $d$-coloring maps\\
  $\OI_d$ & $[n], n \in \N$ & pairs of order-preserving injections and $d$-coloring maps\\
  $\VI_q$ & $\mathbb{F}_q^{\oplus n}$ over a finite field $\mathbb{F}_q$, $n \in \N$ & linear injections.
\end{tabular}
\end{center}

It was firstly observed in \cite{CEF} that the category of $\FI$-modules is equipped with a \emph{shift functor}, which played a central role in proving various fundamental representation theoretic and homological results of $\FI$. In \cite{GL2}, the authors axiomatized certain crucial properties of shift functors and constructed conceptual inductive machinery, generalizing many important results of $\FI$-modules (such as Noetherianity, finiteness of Castelnuovo-Mumford regularity, polynomial growth property, etc.) to modules of abstract categories equipped with shift functors. However, this axiomatic approach only applies to some special examples in the above table such as $\CN$, $\FI$, $\FI_G$ and $\OI$. Stronger machinery working for all examples in the above table is still missing. In particular, although the locally Noetherian property of these examples over commutative Noetherian rings have been established (see \cite{SS}), except for $\CN^d$, $\FI$ and $\OI$, it is unknown whether their finitely generated representations over commutative Noetherian rings have finite Castelnuovo-Mumford regularity.

We note that all these examples have a self-embedding phenomenon. That is, each one is isomorphic to a proper subcategory of itself. This self-embedding functor induces a pullback functor in their module categories, which share many key properties of the shift functor of $\FI$-modules. This observation motivated us to consider generalized shift functors (which we shall call \emph{inductive functors}), and ask whether these generalized shift functors can also provide inductive machinery for proving certain fundamental results of representations. This is indeed the case. That is, the existence of inductive functors makes possible for us to establish generalized inductive machinery (which is significantly improved compared to the one described in \cite{GL2}) to deduce asymptotic behavior of modules and their homology groups of abstract locally Noetherian categories, including all categories in Table 1 except $\VI_q$ as examples. With this inductive machinery, we can show that the existence of inductive functors, together with a few moderate assumptions, guarantees the finiteness of Castelnuovo-Mumford regularity of finitely generated representations. In particular, we deduce that finitely generated representations of categories in Table 1 (the conclusion for $\VI_q$ is still unknown) have finite regularity.

\subsection{Notations}

Throughout this paper we let $\C$ be a \emph{locally finite, locally bounded graded category}. That is, $\C$ is a small category satisfying the following conditions:
\begin{itemize}
\item $\C$ is \emph{locally finite}; that is, for every pair of objects $x, y$, the morphism set $\C(x, y)$ is finite.
\item $\C$ is \emph{directed}; that is, there is a partial ordering $\leqslant$ over $\Ob \C$ such that $\C(x, y) \neq \emptyset$ if and only if $x \leqslant y$.
\item There is a rank function $| \cdot |: \Ob \C \to \N$, where $\N$ is the set of nonnegative integers, such that $|x| \leqslant |y|$ whenever $x \leqslant y$, and $|y| = |x| + 1$ when $y$ \emph{covers} $x$.
\item $\C$ is \emph{locally bounded}; that is, for every $n \in \N$, the set $\{x \in \Ob \C \mid |x| = n \}$ is finite.
\end{itemize}

It is easy to see that all examples in Table 1 are locally finite, locally bounded graded categories, with rank functions being defined by $[n] \mapsto n$, $[n_1] \times \ldots \times [n_d] \to n_1 + \ldots + n_d$, and $\mathbb{F}_q^n \mapsto n$ respectively.

Let $k$ be a commutative Noetherian unital ring. A \emph{representation} $V$ of $\C$ (or a \emph{$\C$-module}) is a covariant functor from $\C$ to the category of all $k$-modules. Morphisms between $\C$-modules are natural transformations. Denote the category of all $\C$-modules by $\CMod$. This is an abelian category containing enough projective objects. In particular, a $\C$-module of the form $k\C(x, -)$ with $x \in \Ob \C$ is projective. We call it a \emph{free} $\C$-module, and denote it by $M(x)$.

Given a $\C$-module $V$, the \emph{value} $V(x)$ of $V$ on an object $x$ in $\C$ is denoted by $V_x$. By abuse of notation, sometimes we regard $V$ as a $k$-module by identifying it with
\begin{equation*}
\bigoplus _{x \in \Ob \C} V_x.
\end{equation*}
The $\C$-module $V$ is \emph{generated in degrees} $\leqslant n$ if it has no proper submodule containing
\begin{equation*}
\bigoplus_{|x| \leqslant n} V_x.
\end{equation*}
The \emph{generating degree} is set to be
\begin{equation*}
\gd(V) = \inf \{ n \in \N \mid V \text{ is generated in degrees } \leqslant n \}
\end{equation*}
or -1 if $V = 0$ by convention. The $\C$-module $V$ is \emph{finitely generated} if there exist a finite set of elements $v_1, \ldots, v_n \in V$ such that no proper submodules of $V$ contains all these elements. Note that free modules are finitely generated. The graded category $\C$ is \emph{locally Noetherian} over $k$ if every finitely generated $\C$-module $V$ is \emph{Noetherian}; that is, submodules of finitely generated $\C$-modules are finitely generated as well. In this case, the category of finitely generated $\C$-modules is abelian.

\subsection{Inductive machinery}

To investigate homological properties and asymptotic behavior of $\C$-modules, we need efficient machinery. In \cite{GL2} we have introduced inductive machinery based on the shift functor introduced in \cite{CEF, CEFN} for $\FI$-modules, uniformly proving quite a few interesting results. Here we generalize this machinery so that it is applicable to many other categories.

A covariant functor $F: \CMod \to \CMod$ is called an \emph{inductive functor of multiplicity $d$} if it satisfies the following requirements:
\begin{itemize}
\item $F$ is exact;
\item there is a natural transformation $\theta: \mathrm{Id}^{\oplus d} \to F$, where $d$ is a positive integer and $\mathrm{Id}$ is the identity functor in $\CMod$;
\item for every object $x \in \Ob \C$, $FM(x)$ is finitely generated, and the natural map $\theta_{M(x)}: M(x)^{\oplus d} \to FM(x)$ is injective;
\item for every nonzero $V \in \CMod$, one has
\begin{equation*}
\gd(DV) = \gd(V) - 1,
\end{equation*}
where $D$ is the cokernel functor induced by natural maps $\theta_V: V^{\oplus d} \to FV$.
\end{itemize}

The above conditions imposed on $F$ are moderate. Although the above definition does not shed any light on how to find or construct such a functor, for many combinatorial categories, in particular for those equipped with a monoid structure, it is often possible to construct a self-embedding functor which induces an inductive functor. Actually, we will use this strategy to define inductive functors for categories $\CN^d$, $\FI^d$, $\OI^d$, $\FI_d$, $\OI_d$.

The existence of natural maps $\theta_V: V^{\oplus d} \to FV$ is very crucial. It actually gives a family of natural maps
\begin{equation*}
\bigoplus_{s \in S} V^s \to V^{\oplus d} \to FV,
\end{equation*}
where $S$ is a subset of $[d]$, $V^s$ means the $s$-th copy of $V$ in $V^{\oplus d}$, and the first map is the inclusion. We denote its kernel and cokernel by $K_SV$ and $D_SV$ respectively. Note that $D_SV = FV$ when $S = \emptyset$ and $D_SV = DV$ when $S = [d]$. Therefore, the collection of functors $D_SV$ interpolate $FV$ and $DV$. Moreover, one has the following exact sequence
\begin{equation*}
0 \to K_SV \to \bigoplus_{s \in S} V^s \to FV \to D_SV \to 0.
\end{equation*}
In practice we are more interested in short exact sequences; that is, $K_SV  = 0$. Correspondingly, a subset $S \subseteq [d]$ is called a \emph{nil} subset for $V$ if $K_SV = 0$. Note that nil subsets for $V$ always exist since when $S = \emptyset$, $K_S V = \bigoplus_{s \in S} V^s = 0$.

Let (P) be a property of certain $\C$-modules such as finitely generated property, finitely presented property, Noetherianity. We say that (P) is
\begin{itemize}
\item \emph{maximal nil dominant} if for every $\C$-module $V$, the assumption that $D_SV$ has (P) with $S$ being a maximal nil subset for $V$ implies that $V$ satisfies (P);
\item \emph{glueable} if in every short exact sequence $0 \to U \to V \to W \to 0$ of $\C$-modules, the assumption that $U$ and $W$ have (P) implies that $V$ satisfies (P).
\end{itemize}

The following theorem describes inductive machinery for exploring asymptotic behavior of $\C$-modules, provided that the locally Noetherianity of $\C$ over $k$ is already known.

\begin{theorem} \label{main theorem I}
Let $\C$ be a locally finite, locally bounded graded category equipped with an inductive functor of multiplicity $d$, and let $k$ be a commutative Noetherian ring. Suppose that $\C$ is locally Noetherian over $k$. Let (P) be a property of certain $\C$-modules which is satisfied by the zero module. Then every finitely generated $\C$-module has (P) if and only if (P) is maximal nil dominant and glueable.
\end{theorem}

The notion of maximal nil dominant may seem artificial to the reader at a first glance. There are two main reasons. Firstly, if $S$ is a nil subset, then there is a short exact sequence
\begin{equation*}
0 \to \bigoplus_{s \in S} V^s \to FV \to D_SV \to 0.
\end{equation*}
Moreover, by the naturality of the map $\bigoplus_{s \in S} V^s \to FV$, every projective resolution $P^{\bullet} \to V \to 0$ induces a short exact sequence of complexes
\begin{equation*}
0 \to \bigoplus_{s \in S} P^{\bullet} \to FP^{\bullet} \to D_SP^{\bullet} \to 0,
\end{equation*}
while the complexes $FP^{\bullet}$ and $D_SP^{\bullet}$ are still exact. Therefore, by imposing moderate conditions on the inductive functor $F$, one may deduce properties of $\bigoplus_{s \in S} V^s$ (for example, finiteness of regularity), and hence properties of $V$, from corresponding properties of $D_SV$. However, if $S$ is not a nil subset, this inductive step in general fails. Secondly, we shall use subsets of $[d]$ to construct a tree of quotient modules of $V$ to carry out the induction, and a maximal nil subset guarantees that this tree is a finite tree, so that our induction can be carried out; see Remark \ref{justification}.

We briefly describe the idea for proving the above theorem. The general strategy is an induction on the generating degree $\gd(V)$, which is finite. Since (P) is maximal nil dominant, it suffices to show that $D_SV$ has (P) where $S$ is a maximal nil subset for $V$. We construct a sequence of quotient maps starting from $D_SV$ and ending at $DV$, in which the kernel of each quotient map is a proper quotient module of $V$. Since $\gd(DV) < \gd(V)$, the induction hypothesis tells us that $DV$ has (P). Since (P) is glueable, it suffices to show that those kernels of quotient maps have (P). We repeat the above step for these kernels, and show that this procedure must stop after finitely many steps, relying on the conditions that $S$ is a maximal nil subset (let us emphasize that this is the only reason why we introduce this notion) and $\C$ is locally Noetherian over $k$.

Actually, for $d = 1$, we proved in \cite{GL2} that the existence of inductive functor guarantees the local Noetherianity of $\C$ over $k$. But at this moment we are not able to generalize this conclusion for the case $d > 1$.

\subsection{A criterion for finiteness of Castelnuovo-Mumford regularity}

As a useful application of the above formalistic theorem, we provide a sufficient criterion for the finiteness of Castelnuovo-Mumford regularity, an important homological invariant in commutative algebra and representation theory of graded rings.

Let $\C$ be a locally finite, locally bounded graded category, and let $\mm$ be the $k$-module
\begin{equation*}
\bigoplus _{x < y} k\C(x, y).
\end{equation*}
By the directed structure of $\C$, the $k$-module $\mm$ is a two-sided ideal of the \emph{category algebra} $k\C$ (defined in \cite{Webb}). Using this ideal, one can define \emph{homology groups} of $\C$-modules as follows. Given a $\C$-module $V$, its zeroth homology group is
\begin{equation*}
H_0(V) = k\C/\mm \otimes_{k\C} V,
\end{equation*}
which is still a $\C$-module. Explicitly, for every object $x$, one has
\begin{equation*}
(H_0(V))_x = V_x / \sum_{y < x} (k\C(y, x) \cdot V_y).
\end{equation*}
Therefore, $H_0(V)$ actually provides a lot of information on generators of $V$. For example, if $(H_0(V))_x$ is nonzero, we can conclude that any set of \emph{homogeneous} generators $\{v_1, \, \ldots, \, v_n \}$ of $V$ must intersect $V_x$ nontrivially, where homogeneous means that for each $v_i$, $i \in [n]$, there is an object $x_i$ such that $v_i \in V_{x_i}$.

Since the functor $k\C/\mm \otimes_{k\C} -$ is right exact, it has left derivative functors. Correspondingly, for $i \geqslant 1$, one defines the $i$-th homology group of $V$ by
\begin{equation*}
H_i(V) = \Tor_i^{k\C} (k\C/\mm, V).
\end{equation*}
Since the category $\CMod$ has enough projectives, computation of homology groups can be carried out by choosing a projective resolution of $V$. Therefore, these homology groups are closely related to homogenous generators of ``syzygies" of $V$.

For any integer $i \geqslant 0$, we define the \emph{$i$-th homological degree} $\hd_i(V)$ of $V$ by
\begin{equation*}
\hd_i(V) = \left\{ \begin{array}{ll}
\sup \left\{ |x| \bigm| (H_i(V))_x \neq 0 \right\} & \mbox{ if } H_i(V)\neq 0,\\
-1 & \mbox{ else. }
\end{array}\right.
\end{equation*}
We define the \emph{regularity} of $V$ by
\begin{equation*}
\reg(V) = \sup\left\{ \hd_i(V)-i \bigm| i\geqslant 0 \right\}.
\end{equation*}
In particular, one can easily deduce from the definition that
\begin{equation*}
\gd(V) = \hd_0(V).
\end{equation*}
For example, one has $\gd(M(x)) = |x|$. Moreover, for any $\C$-module $V$, there exists a surjective homomorphism $P\to V$ where $P$ is of the form $\bigoplus_{i \in I} M(x_i)$ and $\gd(P)=\gd(V)$.

We should mention that our conventions for homological degrees, generating degrees, and regularity differ from some papers; see \cite[Remark 1.9]{GL2}.

Let us pause for while, using the category $\CN^d$ to explain the above construction as well as its potential close relation to a classical construction in commutative algebra. Let $A = k[X_1, \ldots, X_d]$ be the polynomial algebra of $d$ indeterminates over $k$. Note that $A$ is a multi-graded algebra. Moreover, it is not difficult to see that the category of \emph{multi-graded $A$-modules} (see for instance \cite{BH, CDeno}) is equivalent to the category of $\CN^d$-modules. Furthermore, the homology groups, homological degrees, and regularity of a $\CN$-module $V$ defined in our way, coincide with the usual definitions in commutative algebra (using the ideal consisting of polynomials with zero constant in $A$) when viewing $V$ as an $A$-module.

The category $\CN^d$ has a another interpretation in elementary number theory. Let $p_1, \ldots, p_d$ be $d$ distinct prime numbers and let $\mathcal{P}$ be a set of positive integers such that every element in $\mathcal{P}$ is either 1 or can be divided by a prime number $p_i$. Define a partial ordering $\preccurlyeq$ on $\mathcal{P}$ such that $n \preccurlyeq m$ if and only if $n$ divides $m$. For $n \in \mathcal{P}$, its rank $|n| = a_1 + \ldots + a_d$ where $n = p_1^{a_1} \ldots p_d^{a_d}$. Then the reader can check that the poset $\mathcal{P}$, viewed as a graded category, is isomorphic to $\CN^d$.

Now we describe the second main result of this paper.

\begin{theorem} \label{main theorem II}
Let $\C$ be a locally finite, locally bounded graded category equipped with an inductive functor $F$ of multiplicity $d$, and let $k$ be a commutative Noetherian ring. Suppose that $\C$ is locally Noetherian over $k$. If for every $x \in \Ob \C$, $\reg(FM(x)) \leqslant |x|$, then $\reg(V) < \infty$ for every finitely generated $\C$-module $V$.
\end{theorem}

The basic idea to prove this theorem, is to show that under the given conditions, the property of having finite regularity is maximal nil dominant since clearly it is glueable. We will show that this property is nil dominant; that is, for any nil subset $S$ (might not be a maximal nil subset) of $V$, if $\reg(D_SV)$ is finite, so is $\reg(V)$.

\subsection{Finiteness of regularity of representations of certain combinatorial categories}

Let us turn to the combinatorial categories described in Table 1. The locally Noetherian property of them are known. In particular, the locally Noetherian property of $\CN^d$ follows from Hilbert's Basis Theorem since the category of finitely generated $\CN^d$-modules is equivalent to the category of finitely generated multi-graded modules of polynomial rings of $d$ indeterminates. The locally Noetherian property of $\FI^d$, $\OI^d$, $\FI_d$, $\OI_d$, and $\VI_q$ can be proved using the method of Sam and Snowden in \cite{SS} (and Putman-Sam \cite{PS} in the case of $\VI_q$). In contrast, one does not have much information about the regularity of finitely generated modules of these categories. A result \cite[Theorem A]{CE} proved by Church and Ellenberg implies that finitely generated $\FI$-modules $V$ have finite regularity bounded by $\gd(V) + \hd_1(V) - 1$. This result was generalized to $\FI_G$ in \cite{L2} by a different approach. Recently, Sam and Snowden proved in \cite{SS2} that for arbitrary $d \geqslant 1$, finitely generated $\FI_d$-modules $V$ over a field of characteristic 0 have finite regularity, which is bounded by a function in terms of the first several homological degrees of $V$. However, an explicit formula of this function was not described. Furthermore, their approach cannot extend to the situation of positive characteristics as it highly relies on the representation theory of symmetric groups. In this paper we construct various inductive functors for the module categories of these categories (except $\VI_q$), and verify that they have the desired properties described in Theorem \ref{main theorem II}. Therefore, we can uniformly show:

\begin{theorem} \label{main theorem III}
Let $k$ be a commutative Noetherian ring, and let $\C$ be one of the following categories:
\begin{equation*}
\FI^d, \, \OI^d, \, \FI_d, \, \OI_d.
\end{equation*}
Let $V$ be a finitely generated $\C$-module. Then $\reg(V) < \infty$.
\end{theorem}

Unfortunately, although the category of $\VI_q$ has a self-embedding functor, its combinatorial behavior is not good enough to give an inductive functor of multiplicity $d$ in the module category. That is, the last condition in the definition of inductive functors, $\gd(DV) = \gd(V) - 1$, does not hold. We hope that after suitable modification, an updated version of this inductive machinery can handle more categories including $\VI_q$. We also remark that in a recent paper \cite{Nag2}, Nagpal proved that a finitely generated $\VI_q$-module always has finite regularity in the non-describing characteristic case, while Gan and Li gave explicit upper bounds for the regularity in \cite{GL3}.

\subsection{Relative projective modules and minimal resolutions}

Let $k$ be a field, and let $B$ be a graded quotient algebra of $A = k[X_1, \ldots, X_d]$. Then every finitely generated graded $B$-module $V$ (or more generally, every bounded below graded $B$-module) has a minimal projective resolution $P^{\bullet} \to V \to 0$, through which one can investigate asymptotic behavior of syzygies of $V$. In particular, the homology groups of $V$ can be computed via the isomorphism $H_i(V) \cong H_0(P^i)$ for all $i \geqslant 0$. Furthermore, it has been proved by Avramov in \cite[Corollary 2]{Av} that if $B$ is a Koszul algebra, then truncations of finitely generated graded $B$-modules have linear minimal projective resolutions.

Similarly, if $k$ is a field of characteristic 0, then every finitely generated $\FI$-module has a minimal projective resolution. However, for arbitrary commutative Noetherian rings $k$, one cannot expect the existence of such resolutions. To overcome this obstacle, Nagpal proposed a generalization of projective modules for $\FI$ in \cite{Nag}, called \emph{$\sharp$-filtered modules}, \emph{filtered modules}, or \emph{relative projective modules} by different authors. It turns out that these special modules have similar homological behavior as projective $\FI$-modules, and were intensively applied to study homological properties of $\FI$-modules; see \cite{L2, LR, LY1, LY2, Nag, R1, R2}. In this paper we extend this notion to all the categories described in Table 1. \footnote{Actually, the idea and construction work for all small categories, but here we do not pursue this full generality as we are more interested in concrete applications.} In particular, we prove that a finitely generated module $V$ of one of these categories is relative projective if and only if its first homology group vanishes, and if and only if its higher homology groups all vanish. Moreover, for $n \geqslant \reg(V)$ (which is finite by Theorem \ref{main theorem III}), the \emph{truncated module} $\tau_n V$, a submodule of $V$ obtained by letting the values of $V$ on objects $x$ satisfying $|x| < n$ be 0, has a linear minimal relative projective resolution.

\begin{theorem} \label{main theorem IV}
Let $k$ be a commutative Noetherian ring, and let $\C$ be one of the following categories:
\begin{equation*}
\FI^d, \, \OI^d, \, \FI_d, \, \OI_d.
\end{equation*}
Let $V$ be a finitely generated $\C$-module. Then for $n \geqslant \reg(V)$, the truncated module $\tau_n V$ has a resolution $F^{\bullet} \to \tau_n V \to 0$ satisfying the following conditions: for $i \geqslant 0$,
\begin{enumerate}
\item $F^i$ is a relative projective module;
\item $H_0(F^i) \cong H_i(\tau_n V)$;
\item if $F^i \neq 0$, then $F^i$ is generated by its values on objects $x$ with $|x| = n + i$.
\end{enumerate}
In particular, if $H_i(\tau_n V) \neq 0$, then $\hd_i(\tau_n V) = n+i$.
\end{theorem}

Intuitively, given a finitely generated module $V$ of one of the above categories, one can break $V$ into two pieces, one of which is only supported on finitely many objects, and the other piece has a ``Koszul" behavior. When $k$ is a field (for $\OI_d$ and $\OI^d$), or more strictly, a field of characteristic 0 (for $\FI_d$ and $\FI^d$), relative projective modules are precisely projective modules. Therefore, in this case, one concludes from the above theorem that truncated modules $\tau_n V$ for $n \geqslant \reg(V)$ have a linear minimal projective resolutions. That is, they are \emph{Koszul} modules in the classical sense.

\subsection{Questions}

There are still some important questions which we think deserve a further investigation.

\textbf{Question I:} As we mentioned before, for $d = 1$, the existence of inductive functors guarantees the local Noetherian property of $\C$ over a commutative Noetherian ring $k$; see \cite{GL2}. Does this conclusion still hold for $d > 1$? If so, one may not only remove the assumption that $\C$ is locally Noetherian from the statements of Theorems \ref{main theorem I} and \ref{main theorem II}, but also obtain an efficient method to prove the local Noetherianity of infinite combinatorial categories.

\textbf{Question II:} Let $V$ be a finitely generated $\VI_q$-module over a commutative Noetherian ring. Does $V$ has finite regularity? If so, then the conclusion of Theorem \ref{main theorem IV} holds for $\VI_q$ as well. In the non-describing case, this question has been answered by Nagpal in \cite{Nag2}. But in the other case, the answer is still unknown.

\textbf{Question III:} Can people explicitly describe an upper bound of $\reg(V)$ in terms of its first several homological degrees or other related homological invariants? For $\FI$, this has been solved by Church and Ellenberg in \cite[Theorem A]{CE}; that is, $\reg(V) \leqslant \gd(V) + \hd_1(V) - 1$. For $\VI_q$, in the non-describing case, Gan and Li proved the same upper bound in \cite{GL3}. But for $\FI_d$ and $\FI^d$ with $d > 1$, and $\OI^d$ and $\OI_d$ with $d \geqslant 1$, no satisfactory answers is known.

\subsection{Organizations}

The paper is organized as follows. In Section 2 we list a few preliminary results which will be widely used later. The inductive machinery based on inductive functors and the proof of Theorem \ref{main theorem I} are contained in Section 3. In Section 4 we apply this inductive machinery to prove Theorem \ref{main theorem II} by showing the property of having finite regularity is nil dominant. Concrete examples including $\FI_d$ and $\FI^d$ are studied in details in Section 5, where we construct inductive functors for them and show that these functors satisfy conditions in Theorem \ref{main theorem II}, and hence prove Theorem \ref{main theorem III}. Moreover, in this section we also consider relative projective modules, and construct minimal relative projective resolutions of truncated modules, proving Theorem \ref{main theorem IV}.

\section{Preliminaries}

In this section we deduce some preliminary results which will be used extensively in later sections. From now on we let $\C$ be a locally finite, locally bounded graded category whose module category $\CMod$ is equipped with an inductive functor $F$ of multiplicity $d$, $k$ be a commutative Noetherian ring, and $V$ be a $\C$-module. At this moment we do not required that $\C$ be locally Noetherian over $k$. That is, the category of finitely generated $\C$-modules might not be abelian.

Recall that for every subset $S \subseteq [d]$ ($S$ might be the empty set), there is a natural map
\begin{equation*}
\bigoplus_{s \in S} V^s \to V^{\oplus d} \to FV
\end{equation*}
where $V^s$ is the $s$-th copy of $V$ in the direct sum $V^{\oplus d}$, and $K_SV$ and $D_SV$ are the kernel and cokernel of this map respectively.

\begin{lemma} \label{basic lemma}
Let $V$ be a $\C$-module. One has:
\begin{enumerate}
\item $V$ is finitely generated if and only if $\gd(V) < \infty$ and $V_x$ is a finitely generated $k$-module for every $x \in \Ob \C$.
\item For every subset $S \subseteq [d]$, $\gd(D_SV) \leqslant \gd(V) \leqslant \gd(D_SV) + 1$.
\item If $V$ is finitely generated, so are $FV$ and $D_SV$ for all $S \subseteq [d]$.
\item For every short exact sequence $0 \to U \to V \to W \to 0$ of $\C$-modules, there is a long exact sequence
\begin{equation*}
0 \to K_SU \to K_SV \to K_SW \to D_SU \to D_SV \to D_SW \to 0.
\end{equation*}
In particular, if $S$ is a nil subset for $W$, then we have a short exact sequence
\begin{equation*}
0 \to D_SU \to D_SV \to D_SW \to 0.
\end{equation*}
\item If $V$ is a submodule of $P = \bigoplus_{i \in I} M(x_i)$, then for $S \subseteq [d]$, the natural map $\bigoplus_{s \in S} V^s \to FV$ is injective.
\item If the natural map $V^{\oplus d} \to FV$ is injective, and $P^{\bullet} \to V \to 0$ is a projective resolution such that each $P^i$ is a direct sum of free modules, then the complexes $FP^{\bullet} \to FV \to 0$ and $D_SP^{\bullet} \to D_SV \to 0$ are exact for every $S \subseteq [d]$.
\end{enumerate}
\end{lemma}

\begin{proof}
(1): Suppose that $V$ is finitely generated. Then one may choose a finite set of elements $v_1, v_2, \ldots, v_n$ such that $V$ is the only submodule containing all these elements. Viewing $V$ as a $k$-module, we have the following decomposition
\begin{equation*}
V = \bigoplus _{x \in \Ob \C} V_x.
\end{equation*}
Clearly, one can choose a finite set of elements $x_1, \ldots, x_m$ such that
\begin{equation*}
v_i \in \bigoplus_{j \in [m]} V_{x_j}
\end{equation*}
for all $i \in [n]$. Therefore, $\gd(V) \leqslant \max \{ |x_j| \mid j \in [m] \}$ is finite.

Note that the category of finitely generated $\C$-modules has enough projectives. In particular, free modules $M(x)$ are finitely generated for $x \in \Ob \C$. Moreover, since $V$ is finitely generated, one can choose a surjection
\begin{equation*}
P = \bigoplus _{i \in I} M(x_i) \to V
\end{equation*}
with $I$ being a finite set. For each $x \in \Ob \C$, by the locally finite condition, $M(x_i)_x = k\C(x_i, x)$ is a free $k$-module with finite rank. Therefore, $P_x$ is a free $k$-module of finite rank, so $V_x$ is a finitely generated $k$-module.

Conversely, suppose that $\gd(V) < \infty$, and $V_x$ is a finitely generated $k$-module for each $x \in \Ob \C$. Note that $V$ is generated by elements in the direct sum
\begin{equation*}
\bigoplus _{|x| \leqslant \gd(V)} V_x.
\end{equation*}
Since $\gd(V)$ is finite, by the locally bounded condition, the set $\{ x \in \Ob \C \mid |x| \leqslant \gd(V) \}$ is a finite set, so the above direct sum is a finite direct sum. Furthermore, each $V_x$ is a finitely generated $k$-module, and hence a finitely generated $k\C(x, x)$-module. Therefore, for each object $x$ with $|x| \leqslant \gd(V)$, one can choose a finite set $S_x$ of elements in $V_x$, which generate $V_x$ under the action of the monoid $\C(x, x)$. Consequently, elements in the finite set
\begin{equation*}
\bigsqcup _{|x| \leqslant \gd(V)} S_x
\end{equation*}
generate $V$. That is, $V$ is finitely generated.

(2): We may assume that $V \neq 0$ since otherwise the conclusion holds trivially by our convention $\gd(0) = -1$. Note that if $W$ is a quotient module of $V$, then $H_0(W)$ is a quotient module of $H_0(V)$. Consequently, $\gd(W) = \hd_0(W) \leqslant \hd_0(V) = \gd(V)$.

Firstly we prove that $\gd(V) - 1 \leqslant \gd(FV) \leqslant \gd(V)$. The natural map $\theta_V: V^{\oplus d} \to FV$ gives rise to a short exact sequence $0 \to V' \to FV \to DV \to 0$, where $V'$ is the image of $\theta_V$. Therefore, $\gd(V) \geqslant \gd(V')$ and $\gd(FV) \geqslant \gd(DV) = \gd(V)-1$. Also, from the short exact sequence one deduces that
\begin{equation*}
\gd(FV) \leqslant \max \{\gd(V'), \, \gd(DV) \} \leqslant \gd(V).
\end{equation*}

Note that $D_SV$ is a quotient module of $FV$, so $\gd(D_SV) \leqslant \gd(FV) \leqslant \gd(V)$. On the other hand, the commutative diagram
\begin{equation*}
\xymatrix{
\bigoplus_{s \in S} V^s \ar[r] \ar[d] & FV \ar@{=}[d] \ar[r] & D_SV \ar[r] \ar[d] & 0\\
V^{\oplus d} \ar[r] & FV \ar[r] & DV \ar[r] & 0
}
\end{equation*}
tells us that $DV$ is a quotient module of $D_SV$. Thus $\gd(D_SV) \geqslant \gd(DV) = \gd(V) - 1$.

(3): Since $V$ is finitely generated, there is a surjection
\begin{equation*}
\bigoplus_{i \in I} M(x_i)^{a_i} \to V \to 0
\end{equation*}
with $I$ being a finite set and $a_i$ being finite numbers. Applying the exact functor $F$ one gets a surjection
\begin{equation*}
\bigoplus_{i \in I} FM(x_i)^{a_i} \to FV \to 0.
\end{equation*}
Since each $FM(x_i)$ is still finitely generated by the definition of inductive functors, so is $FV$. As we explained before, $D_SV$ are quotient modules of $FV$ for $S \subseteq [d]$, so are finitely generated as well.

(4): Applying $F$ to the given short exact sequence and using the natural transformation $\theta: \mathrm{Id}^{\oplus d} \to F$, we obtain a commutative diagram of short exact sequences
\begin{equation*}
\xymatrix{
0 \ar[r] & \bigoplus_{s \in S} U^s \ar[r] \ar[d] & \bigoplus_{s \in S} V^s \ar[r] \ar[d] & \bigoplus_{s \in S} W^s \ar[r] \ar[d] & 0\\
0 \ar[r] & FU \ar[r] & FV \ar[r] & FW \ar[r] & 0.
}
\end{equation*}
Now the snake lemma gives the wanted conclusion.

(5): The conclusion follows from the commutative diagram
\begin{equation*}
\xymatrix{
\bigoplus_{s \in S}V^s \ar[r] \ar[d] & \bigoplus_{s \in S} P^s \ar[d]\\
FV \ar[r] & FP
}
\end{equation*}
since the maps represented by the top arrow and the right vertical arrow are injective.

(6): Break the projective resolution $P^{\bullet} \to V \to 0$ into short exact sequences $0 \to Z^{i+1} \to P^i \to Z^i \to 0$ (where $Z^0 = V$), $i \geqslant 0$, and apply the conclusions of (4) and (5).
\end{proof}

\begin{remark} \normalfont
All conclusions stated in this lemma are generalized from the corresponding properties of shift functors of $\FI$-modules; see \cite{CEFN, LY1}.

Statement (3) asserts that functors $F$ and $D_S$ in $\CMod$ restrict to endofunctors in the category of finitely generated $\C$-modules. However, if $\C$ is not locally Noetherian over $k$, the kernel functors $K_S$ might not preserve finitely generated $\C$-module.

A $\C$-module, whose value on each object is a finitely generated $k$-module, is called \emph{locally finite}. Since $k$ is Noetherian, and for each $x \in \Ob \C$, the endomorphism algebra $k\C(x, x)$ is finite over $k$, the reader can see that the category of locally finite $\C$-modules is closed under kernels and cokernels. Therefore, this is an abelian category containing the category of finitely generated $\C$-modules. Statement (1) of the previous lemma tells us that a $\C$-module is finitely generated if and only if it is locally finite and has finite generating degree. However, we should remind the reader that this equivalence relies on the condition that $\C$ is locally bounded. If this is not true; that is, there are infinitely many distinct objects in $\C$ having the same rank, then a finitely generated $\C$-module is still locally finite and still has finite generating degree, but the converse statement fails, as explained in the following example.
\end{remark}

\begin{example} \normalfont
Let $\C = \CN^{\infty}$ be the following category: objects are $\bfn = (n_i)_{i \in \N}$ such that all but finitely many $n_i$ are 0; for two objects $\bfn$ and $\mathbf{m}$, there is a unique morphism from $\mathbf{m}$ to $\bfn$ if and only if $m_i \leqslant n_i$ for all $i \in \N$. Clearly, $\C$ is a locally finite graded category whose rank function is given by $|\bfn| = \sum_{i \in \N} n_i$. However, it is not locally bounded since for every positive integer $n$, there are infinitely many distinct objects whose rank is $n$. The category $\CN^{\infty}$ is isomorphic to the following poset $\mathcal{P}$ (viewed as a category in a natural way): elements of $\mathcal{P}$ are positive integers, and $m \preccurlyeq n$ if and only if $m$ divides $n$. Since every positive integer can be written uniquely as a finite product of powers of primes, this decomposition gives rises to an isomorphism from $\mathcal{P}$ to $\CN^{\infty}$ as graded categories.

From the viewpoint of representation theory, these two categories are of special interest, since their module categories are (up to equivalence) subcategories of the category of positively graded modules of the polynomial ring with countably infinitely many indeterminates. Also note that $\CN^{\infty}$ is not locally Noetherian: Let $V$ be the following submodule of $k\CN^{\infty} (\mathbf{0}, -)$ (clearly finitely generated) such that $V_{\bfn} = k$ for all $\bfn \neq \mathbf{0}$, while all linear injections from $V_{\mathbf{m}} = k \to k = V_{\bfn}$ are the identity map. Then $\gd(V) = 1$ is finite, and $V$ is locally finite, but it is not finitely generated.
\end{example}

\section{Inductive machinery}

In this section we establish formal inductive machinery, which is the foundation of all applications described in later sections. From now on we assume that the locally finite, locally bounded graded category $\C$ is locally Noetherian over $k$; that is, the category of finitely generated $\C$-modules is abelian. We also assume that the $\C$-module $V$ is finitely generated.

\subsection{Nil subsets}

Recall that for every subset $S \subseteq [d]$, there is an exact sequence
\begin{equation*}
0 \to K_SV \to \bigoplus_{s \in S} V^s \to FV \to D_SV \to 0.
\end{equation*}
We let $D = D_{[d]}$ and $K = K_{[d]}$.  To simplify the notation, we let
\begin{equation*}
V^S = \bigoplus _{s \in S} V^s.
\end{equation*}
Note that $K_SV = KV \cap V^S$. Therefore, if $S \subseteq T \subseteq [d]$, then
\begin{equation*}
K_SV = KV \cap V^S \subseteq KV \cap V^T = K_TV.
\end{equation*}
Furthermore, for two subsets $S, T \subseteq [d]$, one has
\begin{equation*}
K_SV \cap K_TV = (KV \cap V^S) \cap (KV \cap V^T) = KV \cap (V^S \cap V^T) = KV \cap V^{S \cap T} = K_{S \cap T} V.
\end{equation*}

Let $V$ be a $\C$-module, and let $S\subseteq [d]$. Recall that $S$ is a \emph{nil} subset for $V$ if $K_SV = 0$. It is a \emph{maximal nil} subset for $V$ if $K_S V=0$ and $K_T V\neq 0$ for every $T\subseteq [d]$ which properly contains $S$.

Note that nil subsets always exist since the empty set is a nil subset.

\subsection{Filtrations}

Given any subset $S \subseteq [d]$, there is a strictly increasing chain
\begin{equation*}
S = S_0 \subset S_1 \subset \ldots \subset S_{d-|S|} = [d]
\end{equation*}
such that $|S_{i+1}| = |S_i| + 1$ for $0 \leqslant i \leqslant d - |S| - 1$. By the argument in the previous subsection, there is a corresponding increasing chain
\begin{equation*}
K_S V \subseteq K_{S_1} V \subseteq \ldots \subseteq K_{S_{d-|S|}} V = KV.
\end{equation*}
Note that this chain might not be strictly increasing. That is, it is possible that $K_{S_i} V = K_{S_{i+1}} V$ for a certain $i$.

Now we describe a sequence of quotient maps starting from $D_SV$ and ending at $DV$. For each $S_i$ in the above chain, the exact sequence $0 \to K_{S_i} V \to V^{S_i} \to FV \to D_{S_i} V \to 0$ induces the following commutative diagrams:
\begin{equation*}
\xymatrix{
0 \ar[r] & K_{S_i} V \ar[r] \ar[d] & V^{S_i} \ar[r] \ar[d] & V^{S_i} / K_{S_i} V \ar[r] \ar[d]^{\delta} & 0\\
0 \ar[r] & K_{S_{i+1}} V \ar[r] & V^{S_{i+1}} \ar[r] & V^{S_{i+1}} / K_{S_{i+1}} V \ar[r] & 0
}
\end{equation*}
and
\begin{equation*}
\xymatrix{
0 \ar[r] & V^{S_i} / K_{S_i} V \ar[r] \ar[d]^{\delta} & FV \ar[r] \ar@{=}[d] & D_{S_i} V \ar[r] \ar[d]^{\epsilon} & 0\\
0 \ar[r] & V^{S_{i+1}} / K_{S_{i+1}} V \ar[r] & FV \ar[r] & D_{S_{i+1}} V \ar[r] & 0.
}
\end{equation*}
Applying the snake lemma to the second commutative diagram, we know that $\delta$ is injective, and the kernel of $\epsilon$ is isomorphic to the cokernel of $\delta$, which is
\begin{equation*}
(V^{S_{i+1}} / K_{S_{i+1}} V) / (V^{S_i} / K_{S_i} V) \cong V^{S_{i+1}} / (V^{S_i} + K_{S_{i+1}}V);
\end{equation*}
while applying the snake lemma to the first commutative diagram, we know that the cokernel of $\delta$ is isomorphic to
\begin{equation*}
(V^{S_{i+1}} / V^{S_i}) / (K_{S_{i+1}} V / K_{S_i} V) \cong V/ (K_{S_{i+1}} V / K_{S_i} V),
\end{equation*}
a quotient module of $V$. Thus we proved the following lemma:

\begin{lemma} \label{filtration}
Let $V$ be a finitely generated $\C$-module, and $S$ be a subset of $[d]$. Then there is a sequence of quotient maps
\begin{equation*}
D_SV \to D_{S_1} V \to \ldots \to D_{S_{d - |S| - 1}} V \to D_{S_{d - |S|}} V = DV
\end{equation*}
such that for $0 \leqslant i \leqslant d - |S| - 1$, the kernel of $D_{S_i} V \to D_{S_{i+1}} V$ is isomorphic to
\begin{equation*}
V^{S_{i+1}} / (V^{S_i} + K_{S_{i+1}}V) \cong V/ (K_{S_{i+1}} V / K_{S_i} V).
\end{equation*}
\end{lemma}

The following example illustrate our construction.

\begin{example} \normalfont \label{example}
Let $A = k[X, Y]$, $I = (XY)$, and $V = A/I$. As we mentioned before, the graded $A$-module corresponds to a $\C = \CN^2$-module of the following form:
\begin{equation*}
\begin{matrix}
\ldots & \to & \ldots & \to & \ldots & \to & \ldots & \to & \ldots\\
\uparrow & & \uparrow & & \uparrow & & \uparrow & & \ldots \\
k & \to & 0 & \to & 0 & \to & 0 & \to & \ldots\\
\uparrow & & \uparrow & & \uparrow & & \uparrow & & \ldots \\
k & \to & 0 & \to & 0 & \to & 0 & \to & \ldots\\
\uparrow & & \uparrow & & \uparrow & & \uparrow & & \ldots\\
k & \to & 0 & \to & 0 & \to & 0 & \to & \ldots\\
\uparrow & & \uparrow & & \uparrow & & \uparrow & & \ldots\\
k & \to & k & \to & k & \to & k & \to & \ldots
\end{matrix}
\end{equation*}
where all maps from $k$ to $k$ are identity.

The category $\C$ has two self-embedding functors defined by
\begin{equation*}
\iota_1: [m] \times [n] \mapsto [m+1] \times [n]
\end{equation*}
and
\begin{equation*}
\iota_2: [m] \times [n] \mapsto [m] \times [n+1],
\end{equation*}
which induce two shift functors $\Sigma_1$ and $\Sigma_2$ in $\CMod$ via pull-back. We will show that the functor $\Sigma = \Sigma_1 \oplus \Sigma_2$ is an inductive functor. Thus we can describe $KV$ and $\Sigma V$ as follows:
\begin{align*}
KV = &
\begin{matrix}
\ldots & \to & \ldots & \to & \ldots & \to & \ldots & \to & \ldots\\
\uparrow & & \uparrow & & \uparrow & & \uparrow & & \ldots \\
k & \to & 0 & \to & 0 & \to & 0 & \to & \ldots\\
\uparrow & & \uparrow & & \uparrow & & \uparrow & & \ldots \\
k & \to & 0 & \to & 0 & \to & 0 & \to & \ldots\\
\uparrow & & \uparrow & & \uparrow & & \uparrow & & \ldots\\
k & \to & 0 & \to & 0 & \to & 0 & \to & \ldots\\
\uparrow & & \uparrow & & \uparrow & & \uparrow & & \ldots\\
0 & \to & 0 & \to & 0 & \to & 0 & \to & \ldots
\end{matrix}
& \bigoplus \quad &
\begin{matrix}
\ldots & \to & \ldots & \to & \ldots & \to & \ldots & \to & \ldots\\
\uparrow & & \uparrow & & \uparrow & & \uparrow & & \ldots \\
0 & \to & 0 & \to & 0 & \to & 0 & \to & \ldots\\
\uparrow & & \uparrow & & \uparrow & & \uparrow & & \ldots \\
0 & \to & 0 & \to & 0 & \to & 0 & \to & \ldots\\
\uparrow & & \uparrow & & \uparrow & & \uparrow & & \ldots\\
0 & \to & 0 & \to & 0 & \to & 0 & \to & \ldots\\
\uparrow & & \uparrow & & \uparrow & & \uparrow & & \ldots\\
0 & \to & k & \to & k & \to & k & \to & \ldots
\end{matrix}
\end{align*}
\begin{align*}
\Sigma V = &
\begin{matrix}
\ldots & \to & \ldots & \to & \ldots & \to & \ldots & \to & \ldots\\
\uparrow & & \uparrow & & \uparrow & & \uparrow & & \ldots \\
0 & \to & 0 & \to & 0 & \to & 0 & \to & \ldots\\
\uparrow & & \uparrow & & \uparrow & & \uparrow & & \ldots \\
0 & \to & 0 & \to & 0 & \to & 0 & \to & \ldots\\
\uparrow & & \uparrow & & \uparrow & & \uparrow & & \ldots\\
0 & \to & 0 & \to & 0 & \to & 0 & \to & \ldots\\
\uparrow & & \uparrow & & \uparrow & & \uparrow & & \ldots\\
k & \to & k & \to & k & \to & k & \to & \ldots
\end{matrix}
& \bigoplus \quad &
\begin{matrix}
\ldots & \to & \ldots & \to & \ldots & \to & \ldots & \to & \ldots\\
\uparrow & & \uparrow & & \uparrow & & \uparrow & & \ldots \\
k & \to & 0 & \to & 0 & \to & 0 & \to & \ldots\\
\uparrow & & \uparrow & & \uparrow & & \uparrow & & \ldots \\
k & \to & 0 & \to & 0 & \to & 0 & \to & \ldots\\
\uparrow & & \uparrow & & \uparrow & & \uparrow & & \ldots\\
k & \to & 0 & \to & 0 & \to & 0 & \to & \ldots\\
\uparrow & & \uparrow & & \uparrow & & \uparrow & & \ldots\\
k & \to & 0 & \to & 0 & \to & 0 & \to & \ldots
\end{matrix}
\end{align*}
while $DV = 0$. Using the language of $A$-modules, $KV = (\bar{Y}) \oplus (\bar{X})$, $\Sigma V \cong A/(Y) \oplus A/(X)$.

Consider the following increasing chain: $S_0 = \emptyset \subset S_1 = \{1 \} \subset S_2 = \{1, 2 \}$. Then the increasing chain of $K_{S_i} V$ is
\begin{equation*}
K_{S_0} V = 0 \subset K_{S_1} V \cong (\bar{Y}) \subset KV = (\bar{Y}) \oplus (\bar{X}),
\end{equation*}
and
\begin{equation*}
D_{S_0} V = \Sigma V \cong A/(Y) \oplus A/(X) \to D_{S_1} V \cong A/(X) \to DV = 0
\end{equation*}
is the desired sequence of quotient maps. Note that these quotient maps are split since $\Sigma$ is a direct sum of two functors. This is not the case in general.
\end{example}

The sequence of quotient maps described in Lemma \ref{filtration} might not be optimal, in the sense that for some $i$, the kernel of $D_{S_{i+1}}V \to D_{S_i} V$ is isomorphic to $V$, or equivalently, $K_{S_i}V = K_{S_{i+1}}V$. However, if we start from a maximal nil subset for $V$, then the chain $K_{S_i}V \subseteq K_{S_{i+1}} V$ is strictly increasing, and hence the kernels of quotient maps $D_{S_i} V \to D_{S_{i+1}}V$ are always proper quotients of $V$; that is, these kernels can not be isomorphic to $V$ (see Lemma 3.1). This is crucial for us to carry out the inductive machinery.

\begin{lemma} \label{jump}
Let $V$ be a $\C$-module, and let $S$ be a maximal nil subset for $V$. Let $T$ and $T'$ be subsets of $[d]$ such that $S \subseteq T \subseteq T'$. If $K_TV = K_{T'}V$, then $T = T'$.
\end{lemma}
\begin{proof}
Suppose, on the contrary, that $T\neq T'$. Then there exists an element $t' \in T' \setminus T$. Let $S' = S \cup \{t'\}$. Since $S$ is a nil maximal subset for $V$, one has $K_{S'}V\neq 0$. But
\begin{equation*}
K_{S'}V = K_{S'}V\cap K_{T'}V =K_{S'}V\cap K_T V = K_{S'\cap T} V = K_S V = 0,
\end{equation*}
a contradiction.
\end{proof}

By this lemma, starting from a maximal nil subset $S$ of $V$, a strictly increasing chain
\begin{equation*}
S = S_0 \subset S_1 \subset S_2 \subset \ldots \subset S_{d-|S|} = [d]
\end{equation*}
of subsets (note that the strictly increasing condition implies that $|S_{i+1}| = |S_i| + 1$) induces a strictly increasing sequence of submodules of $KV$:
\begin{equation*}
K_S V \subset K_{S_1} V \subset \ldots \subset K_{S_{d-|S|}} V = KV.
\end{equation*}
Correspondingly, in the sequence of quotient maps described in Lemma \ref{filtration}
\begin{equation*}
D_SV \to D_{S_1} V \to \ldots \to D_{S_{d - |S| - 1}} V \to D_{S_{d - |S|}} V = DV,
\end{equation*}
all kernels are nonzero. We denote the kernel of $D_{S_i} V \to D_{S_{i+1}} V$ by $V^{(i)}$, which is isomorphic to
\begin{equation*}
V^{S_{i+1}} / (V^{S_i} + K_{S_{i+1}}V) \cong V/ (K_{S_{i+1}} V / K_{S_i} V),
\end{equation*}
a proper quotient module of $V$, and call $V^{(i)}$ a \emph{child} of $V$ if $V^{(i)}\neq 0$.

\subsection{A finite tree}

Now we construct a tree of \emph{descendants} of $V$, and show that it is a finite tree. We will use this tree as a bookkeeping device for our inductive machinery.

Recall that we defined a collections of children for $V$ in the previous subsection. Now choose a maximal nil subset for each child of $V$ and repeat the procedure described in the previous subsection, we obtain the children of each child of $V$. Recursively, we get a tree of quotient modules, depicted as below:
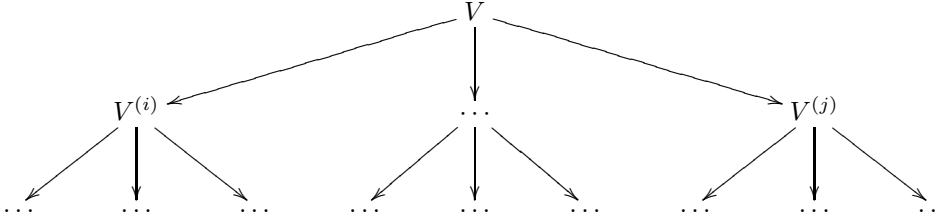
\begin{equation*}
\xymatrix{
 & & & & V \ar[llld] \ar[d] \ar[rrrd] & & & & \\
 & V^{(i)} \ar[ld] \ar[d] \ar[rd] & & & \cdots \ar[ld] \ar[d] \ar[rd] & & & V^{(j)} \ar[ld] \ar[d] \ar[rd] & \\
 \ldots & \ldots & \ldots & \ldots & \ldots & \ldots & \ldots & \ldots & \ldots
}
\end{equation*}
The arrows in this tree are the quotient maps. We denote this tree by $\mathcal{T}(V)$. It has a natural poset structure with $V$ the minimal element. Each $\C$-module $U$ in $\mathcal{T}(V)$ has at most $d$ children, and the quotient map from $U$ to any child of $U$ has nonzero kernel.

\begin{example} \normalfont
Let us continue Example \ref{example}. The module $V = k[X, Y]/(XY)$ has a unique maximal nil subset, the empty set. Therefore, it has two children, which are $k[X, Y]/(Y)$ and $K[X,Y]/(X)$ respectively. Each of these two children has a unique maximal nil subset, which are $\{1 \}$ and $\{2 \}$ respectively. An explicit computation tells us that they have no children. Therefore, we can describe $\mathcal{T}(V)$ as follows:
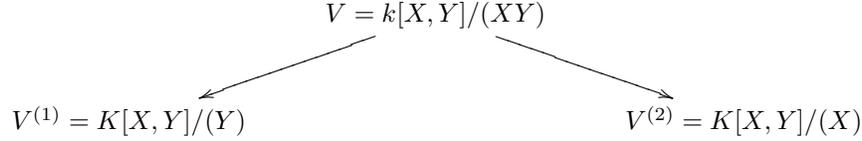
\begin{equation*}
\xymatrix{
 & V = k[X, Y]/(XY) \ar[dl] \ar[dr] & \\
V^{(1)} = K[X, Y]/(Y) & & V^{(2)} = K[X, Y]/(X)
}
\end{equation*}
\end{example}

\begin{lemma} \label{lifting in tree}
Let $U$ be a $\C$-module in $\mathcal{T}(V)$, and (P) be a property of certain $\C$-modules which is glueable and maximal nil dominant. If all children of $U$ and $DU$ satisfy (P), then $U$ satisfies (P) as well.
\end{lemma}

\begin{proof}
Let $S$ be a maximal nil subset. We have the sequence of natural surjective homomorphisms
\begin{equation*}
D_{S} V = D_{S_0} V \to D_{S_1} V \to \cdots \to D_{S_{d-|S|}} V = DV
\end{equation*}
and every kernel of a homomorphism in this sequence has (P). Since $DV$ has (P) and (P) is glueable, it follows successively that $D_{S_{d-|S|-1}} V, \ldots, D_{S_1} V$ and $D_S V$ have (P). But (P) is maximal nil dominant, so $V$ has (P) as well.
\end{proof}

Let $U$ and $U'$ be $\C$-modules in the tree $\mathcal{T}(V)$. We say that $U'$ is a \emph{descendent} of $U$ if there is a directed path in $\mathcal{T}(V)$ from $U$ to $U'$.

We shall use the following simple observation:

\begin{lemma} \label{tree finite}
Let $V$ be a finitely generated $\C$-module. Then $\mathcal{T}(V)$ is a finite tree if and only if every chain in $\mathcal{T}(V)$ has finite length.
\end{lemma}

\begin{proof}
One direction is clear. We show the other direction. Suppose that $\mathcal{T}(V)$ is an infinite tree. We construct an infinite chain in the following way. Let $V$ be the first element in the chain. Since $V$ has finitely many children and infinitely many descendants, there must be a child $V'$ of $V$ with infinitely many descendants. Let $V'$ be the second element in the chain, and repeat the procedure. Recursively, we get an infinite chain.
\end{proof}

Now we show that $\mathcal{T}(V)$ is a finite tree.

\begin{lemma} \label{T(V) is finite}
Let $V$ be a finitely generated $\C$-module. Then $\mathcal{T}(V)$ is a finite tree.
\end{lemma}

\begin{proof}
By Lemma \ref{tree finite}, we have to show that every chain in $\mathcal{T}(V)$ has finite length. To this end, suppose on the contrary that there is an infinite chain
\begin{equation*}
U_0 \to U_1 \to U_2 \to \cdots.
\end{equation*}
For each integer $n\geqslant 1$, let $Z_n \subset U_0$ be the kernel of the composition of the quotient maps $U_0 \to U_1 \to \cdots \to U_n$. Since each quotient map $U_{n-1} \to U_n$ has a nonzero kernel, we have a strictly increasing chain $Z_1 \subset Z_2 \subset \cdots$ in $U_0$. But $U_0$ is a finitely generated $\C$-module since it is a quotient module of $V$. Since $\C$ is supposed to be locally Noetherian, we know that $U_0$ is Noetherian, so it cannot have an infinite strictly increasing chain of submodules.
\end{proof}

\begin{remark} \normalfont \label{justification}
We had defined maximal nil subsets for $V$ and used a maximal nil subset to construct a sequence of quotient maps starting from $D_SV$ and ending at $DV$. This construction seems artificial and unnatural compared to the analogue construction starting from $D_{\emptyset} V = \Sigma V$. The above lemma justifies our choice. Indeed, if the sequence of quotient maps starts from $D_{\emptyset} V = \Sigma V$, then it is very possible that $V$ appears as the kernel of a certain quotient map. That is, $V$ becomes a child of itself. Repeating this construction, the reader can see that $V$ appears repeatedly, so the tree has an infinite chain each member of which is $V$. In particular, the tree cannot be finite. Also note that the locally Noetherian property of $\C$ over $k$ is crucial for proving the finiteness of $\mathcal{T}(V)$.
\end{remark}

\subsection{Proof of Theorem \ref{main theorem I}}

We are now ready to prove Theorem \ref{main theorem I}.

\begin{theorem}
Let $\C$ be a locally finite, locally bounded graded category equipped with an inductive functor of multiplicity $d$, and let $k$ be a commutative Noetherian ring. Suppose that $\C$ is locally Noetherian over $k$. Let (P) be a property of certain $\C$-modules which is satisfied by the zero module. Then every finitely generated $\C$-module has (P) if and only if (P) is maximal nil dominant and glueable.
\end{theorem}

\begin{proof}
We want to prove that every finitely generated $\C$-module $V$ has (P). We use induction on $\gd(V)$. The base case $\gd(V)=-1$ ($V = 0$) is already implied by the given condition. Suppose that $\gd(V)\geqslant 0$.

We now show that $V$ has (P). Suppose $U$ is any $\C$-module in $\mathcal{T}(V)$. Then $U$ is a quotient module of $V$, so it is finitely generated and $\gd(U)\leqslant \gd(V)$. By the properties of inductive functor, we have $\gd(DU) < \gd(U) \leqslant \gd(V)$. Hence, by induction hypothesis, we know that $DU$ has property (P).

Let $N$ be the number of layers in $\mathcal{T}(V)$ which is finite by Lemma \ref{T(V) is finite}, where $V$ is at the first layer, the children of $V$ are at the second layer, and so on. Since the $\C$-modules at the $N$-th layer have no child, we deduce from Lemma \ref{lifting in tree} that they have property (P). If a $\C$-module is at the $(N-1)$-th layer, then its children (if any) are at the $N$-th layer. Applying Lemma \ref{lifting in tree} again, we deduce that the $\C$-modules at the $(N-1)$-th layer have property (P). Recursively, we deduce that every $\C$-module in $\mathcal{T}(V)$, including $V$, has property (P).
\end{proof}

\begin{remark} \normalfont
A careful reader can observe that we do not exhaust the full power of inductive functors to prove Theorem \ref{main theorem I}. Indeed, the assumption that $M(x)^{\oplus d} \to FM(x)$ are injective for all $x \in \Ob \C$ was not used, and the assumption that $\gd(DV) = \gd(V) - 1$ for nonzero $V$ is unnecessarily strong.
There are two reasons why we impose these strong requirements on the inductive functors. Firstly, these conditions are satisfied by many interesting examples such as shift functors of $\FI$ and $\OI$ (see \cite{GL2}); secondly, without these strong requirements, although we can still prove Theorem \ref{main theorem I} (which means that we can find more inductive functors), the behavior of inductive functors is less nice, so properties (P) which are nil maximal dominant become sparse accordingly, and hence the applications in practice are more restricted. In conclusion, we have to carefully keep a subtle balance between theoretic formalism and real applications. The reader will see this point of view in next section.
\end{remark}

\section{Homological degrees and Castelnuovo-Mumford regularity}

The conclusion of Theorem \ref{main theorem I} enable us to show that every finitely generated $\C$-module satisfies a property (P) formally via checking (P) is maximal nil dominant and glueable. As it is routine to verify the glueable condition, the challenge is to show that (P) is maximal nil dominant. A basic idea is to use the short exact sequence $0 \to V^S \to FV \to D_SV \to 0$ and the naturality of the map $V^S \to FV$. In this section we use this idea to provide an important application of this formal procedure, providing a sufficient criterion for the finiteness of regularity of finitely generated $\C$-modules.

Let us briefly recall the definitions of \emph{homology groups}, \emph{homological degrees}, and \emph{Castelnuovo-Mumford regularity}. The \emph{category algebra} $k\C$ of $\C$ was define in \cite{Webb}. It is a free $k$-module spanned by all morphisms in $\C$, whose product is given by the following rule: for two morphisms $\alpha$ and $\beta$ in $\C$
\begin{equation*}
\alpha \cdot \beta =
\begin{cases}
\alpha \circ \beta & \text{ if $\alpha$ and $\beta$ can be composed},\\
0 & \text{ otherwise}.
\end{cases}
\end{equation*}
This is an associative algebra without identity. It has a two-sided ideal $\mathfrak{m}$ spanned by
\begin{equation*}
\bigsqcup_{x < y} \C(x, y).
\end{equation*}
Note that
\begin{equation*}
k\C/ \mathfrak{m} = \bigoplus_{x \in \Ob \C} k\C(x, x),
\end{equation*}
both as a two-sided $\C$-module and as a $k$-algebra.

Given a finitely generated $\C$-module $V$ and $i \in \N$, the $i$-th homology group of $V$ is defined to be
\begin{equation*}
H_i(V) = \Tor^i_{k\C} (\mathfrak{m}, V)
\end{equation*}
which is still a finitely generated $\C$-module since $\C$ is locally Noetherian over $k$. The $i$-th homological degree is
\begin{equation*}
\hd_i(V) = \sup \{ |x| \mid (H_i(V))_x \neq 0 \},
\end{equation*}
which is a finite number by (1) of Lemma \ref{basic lemma}. The Castelnuovo-Mumford regularity (regularity for short) is
\begin{equation*}
\reg(V) = \sup \{ \hd_i(V) - i \mid i \in \N \}.
\end{equation*}
Note that $\reg(V) \geqslant \gd(V)$.

We introduce the following notion as a basic technical tool to study the above homological invariants.

\begin{definition} \label{adaptable projective resolution}
A projective resolution $\ldots \to P^2 \to P^1 \to P^0 \to V \to 0$ is \emph{adaptable} if the following two conditions are satisfied:
\begin{itemize}
\item each $P^i$ is a finite direct sum of free modules;
\item $\gd(P^i) = \gd(Z^i)$ for $i \geqslant 0$, where $Z^0=V$, $Z^1$ is the kernel of $P^0 \to V$, and $Z^i$ is the kernel of $P^{i-1} \to P^{i-2}$ for $i \geqslant 2$.
\end{itemize}
\end{definition}

Let $P^{\bullet} \to V \to 0$ be a fixed adaptable projective resolution (the reader can check the existence of such resolutions), and $S \subseteq [d]$ be a nil subset for $V$. Applying the inductive functor $F$ and $D_S$, we obtain two exact complexes $FP^{\bullet} \to FV \to 0$ and $D_SP^{\bullet} \to D_SV \to 0$ by (6) of Lemma \ref{basic lemma}. Note that these two complexes in general are not projective resolutions. However, we can still deduce useful information from them.

Firstly we consider free modules. Note that every subset $S \subseteq [d]$ is a nil subset for a free module.

\begin{lemma}
Let $S$ be a subset of $[d]$. One has $\reg(D_SM(x)) \leqslant \reg(FM(x))$ for all $x \in \Ob \C$.
\end{lemma}

\begin{proof}
For $i = 0$, one has $\hd_0(FM(x)) \geqslant \hd_0(D_SM(x))$ since $D_SM(x)$ is a quotient module of $FM(x)$. For $i \geqslant 2$ the short exact sequence $0 \to M(x)^S \to FM(x) \to D_SM(x) \to 0$ and the long exact sequence
\begin{equation*}
\ldots \to H_1(M(x)) = 0 \to H_1 (FM(x)) \to H_1(D_SM(x)) \to H_0(M(x)) \to H_0(FM(x)) \to H_0(D_SM(x)) \to 0
\end{equation*}
associated with it tell us $\hd_i(FM(x)) \cong \hd_i(D_SM(x))$. Therefore, it suffices to show that
\begin{equation*}
\hd_1(D_SM(x)) \leqslant \reg(FM(x)) + 1.
\end{equation*}
This is true because
\begin{equation*}
\hd_1(D_S M(x)) \leqslant \max \{ \hd_1 (FM(x)), \, \gd(M(x)) \} \leqslant \max \{ \hd_1 (FM(x)), \, \gd(FM(x)) + 1 \} \leqslant \reg(FM(x)) + 1.
\end{equation*}
\end{proof}

\begin{proposition} \label{bounding hd} \normalfont
Let $V$ be a finitely generated $\C$-module and let $S \subseteq [d]$ be a nil subset for $V$. Suppose that $\reg(FM(x)) \leqslant |x|$ for all $x \in \Ob \C$. Then:
\begin{enumerate}
\item For all $i \in \N$, $\hd_i (V) \leqslant \reg(D_SV) + i + 1$, and in particular $\reg(V) \leqslant \reg(D_SV) + 1$.
\item For all $i \in \N$, $\hd_i(D_SV) \leqslant \max \{ \hd_j(V) + i -j \mid 0 \leqslant j \leqslant i \}$, and in particular $\reg(D_SV) \leqslant \reg(V)$.
\end{enumerate}
\end{proposition}

\begin{proof}
The conclusion holds clearly for $V= 0$, so we let $V$ be a nonzero module.

(1): As explained before, the adaptable projective resolution $P^{\bullet} \to V \to 0$ induces an exact complex $D_SP^{\bullet} \to DV \to 0$. By the previous lemma, for each term $P^j$ appearing in $P^{\bullet} \to V \to 0$, one has $\reg(D_SP^j) \leqslant \reg(FP^j)$.

Recall that $Z^0 = V$, $Z^1$ is the kernel of $P^0 \to V$, and $Z^i$ is the kernel of $P^{i-1} \to P^{i-2}$ for $i \geqslant 2$. We show that
\begin{align*}
\gd(Z^i) \leqslant \reg(D_SV) + i + 1;\\
\reg(D_SZ^i) \leqslant \reg(D_SV) + i.
\end{align*}
The conclusion of (1) follows from the first inequality since $\hd_i(V) \leqslant \gd(Z^i)$.

For $i = 0$, these inequalities hold clearly since
\begin{equation*}
\gd(V) \leqslant \gd(D_SV) + 1 \leqslant \reg(D_SV) + 1.
\end{equation*}
Suppose that they hold for all $j$ with $j \leqslant i$ and let $j = i + 1$. Consider the short exact sequences $0 \to Z^{i+1} \to P^i \to Z^i \to 0$ and $0 \to D_SZ^{i+1} \to D_SP^i \to D_SZ^i \to 0$ (the exactness follows from (4) of Lemma \ref{basic lemma}). We get
\begin{align*}
\reg(D_SZ^{i+1}) & \leqslant \max \{ \reg(D_SP^i), \, \reg(D_SZ^i) + 1 \}\\
& \leqslant \max \{ \reg(FP^i), \, \reg(D_SV) + i + 1 \} & \text{the previous Lemma and induction hypothesis} \\
& \leqslant \max \{ \gd(P^i), \, \reg(D_SV) + i + 1 \} & \text{the given assumption} \\
& \leqslant \max \{ \gd(Z^i), \, \reg(D_SV) + i + 1 \} & \text{since $\gd(P^i) = \gd(Z^i)$}\\
& \leqslant \max \{ \reg(D_SV) + i + 1, \, \reg(D_SV) + i + 1 \} & \text{by the induction hypothesis}\\
& = \reg(D_SV) + i+1.
\end{align*}
Therefore,
\begin{equation*}
\gd(Z^{i+1}) \leqslant \gd(D_SZ^{i+1}) + 1 \leqslant \reg(D_SZ^{i+1}) + 1 \leqslant \reg(D_SV) + i + 2.
\end{equation*}
The conclusion follows by induction.

(2): The inequality holds for $i = 0$. Suppose that it holds for $j$ with $j \leqslant i$, and consider the $(i+1)$-th homological degrees. By the short exact sequence $0 \to D_SZ^1 \to D_SP^0 \to D_SV \to 0$ we have:
\begin{align*}
\hd_{i+1}(D_SV)  & \leqslant \max \{ \hd_i(D_SZ^1), \, \hd_{i+1}(D_SP^0) \} \\
& \leqslant \max \{\hd_i(D_SZ^1), \, \gd(D_SP^0) + i + 1 \} & \text{ since $\reg(D_SP^0) \leqslant \gd(P^0)$}\\
& \leqslant \max \{ \{ \hd_j(Z^1) + i -j \mid 0 \leqslant j \leqslant i \} \cup \{ \gd(P^0) + i + 1 \} \} & \text{ the induction hypothesis on $Z^1$}
\end{align*}
Note that $\gd(P^0) = \gd(V)$, $\hd_j(Z^1) = \hd_{j+1}(V)$ for $j \geqslant 1$, and
\begin{equation*}
\gd(Z^1) \leqslant \max \{\gd(P^0), \, \hd_1(V) \}.
\end{equation*}
Putting them together, we obtain
\begin{align*}
\hd_{i+1}(D_SV)  & \leqslant \max \{ \gd(V) + i, \, \hd_1(V) + i, \, \hd_2(V) + i - 1, \, \ldots, \, \hd_{i+1}(V), \,  \gd(V) + i + 1\} \\
& = \max \{ \hd_j(V) + i + 1 -j \mid 0 \leqslant j \leqslant i + 1 \}.
\end{align*}
Therefore, the inequality holds for $i+1$.
\end{proof}

\begin{remark} \normalfont
The condition that $\gd(DV) = \gd(V) - 1$ imposed on inductive functors, which is unnecessarily strong for proving Theorem \ref{main theorem I}, plays a vital role in the proof of this proposition.
\end{remark}

An straightforward corollary is:

\begin{corollary} \label{a corollary}
Let $V$ be a finitely generated $\C$-module. Then
\begin{equation*}
\reg(FV) \leqslant \reg(V) \leqslant \reg(FV) + 1.
\end{equation*}
\end{corollary}

\begin{proof}
Let $S$ be the empty set in the conclusion of Proposition \ref{bounding hd}.
\end{proof}

Now we are ready to prove Theorem \ref{main theorem II}.

\begin{theorem}
Let $\C$ be a locally finite, locally bounded graded category equipped with an inductive functor $F$ of multiplicity $d$, and let $k$ be a commutative Noetherian ring. Suppose that $\C$ is locally Noetherian over $k$. If for every $x \in \Ob \C$, $\reg(FM(x)) \leqslant |x|$, then $\reg(V) < \infty$ for every finitely generated $\C$-module $V$.
\end{theorem}

\begin{proof}
Clearly, the property of having finite regularity is glueable. Statement (1) of Proposition \ref{bounding hd} asserts that it is also max nil dominant since we can choose $S$ in that proposition to be a maximal nil subset. Now apply the conclusion of Theorem \ref{main theorem I}.
\end{proof}

Actually, we can see from Statement (1) of Proposition \ref{bounding hd} that the property of having finite regularity is not only maximal nil dominant, but nil dominant. That is, the nil subset $S \subseteq [d]$ need not be maximal. Actually, the only reason that we choose $S$ to be a maximal nil subset has been explained in Remark 3.9: maximal nil subsets guarantees the finiteness of the tree $\mathcal{T}(V)$ so that we can carry out the inductive machinery.

\section{Applications}

In this section we apply the theoretic results developed in previous sections to those concrete combinatorial categories in Table 1, except $\VI_q$. We briefly describe their definitions, construct an inductive functor for each of them, and verify that their inductive functors satisfy all conditions required for Theorems \ref{main theorem I} and \ref{main theorem II} (actually, these inductive functors we constructed satisfy extra nice properties such as preserving projective modules). Furthermore, we will construct minimal linear relative projective resolutions for truncated modules of these categories, proving Theorem \ref{main theorem IV}.

\subsection{The category $\CN^d$}

Recall that objects of $\C = \CN^d$ are $d$-tuples $\bfn = (n_1, \ldots, n_d)$ with $n_i \in \N$, and there is a unique morphism from $\mathbf{m} \to \bfn$ if $m_i \leqslant n_i$ for each $i \in [d]$. A representation $V$ of $\CN^d$ gives rise to a $k$-module
\begin{equation*}
\bigoplus _{\bfn \in \Ob \C} V_{\bfn}
\end{equation*}
which can be viewed as a \emph{positively multigraded module} of the polynomial ring $A = k[X_1, \ldots, X_d]$.

There are $d$ faithful functors $\C \to \C$ defined as follows:
\begin{equation*}
\iota_i: \C \to \C, \quad (n_1, \ldots, n_{i-1}, n_i, n_{i+1}, \ldots, n_d) \mapsto (n_1, \ldots, n_{i-1}, n_i+1, n_{i+1}, \ldots, n_d)
\end{equation*}
which induce $d$ \emph{shift functors} $\Sigma_i: \CMod \to \CMod$ by sending a $\C$-module $V$ to $V \circ \iota_i$. Let
\begin{equation*}
\boldsymbol{\Sigma} = \Sigma_1 \oplus \ldots \oplus \Sigma_d.
\end{equation*}
We shall show that this is an inductive functor for $\CMod$ satisfying all the requirements in Theorems \ref{main theorem I} and \ref{main theorem II}.

\begin{lemma} \label{lemma for N^d}
The above functor $\Sigma$ is an inductive functor. Moreover, for $\bfn \in \Ob \C$, one has
\begin{equation*}
\boldsymbol{\Sigma} M(\bfn) \cong \Big{(} \bigoplus_{n_i \geqslant 1} M(\bfn - \mathbf{1}_i) \Big{)} \oplus \Big{(} \bigoplus_{n_i = 0} M(\bfn) \Big{)},
\end{equation*}
where $\mathbf{1}_i = (0, \ldots, 0, 1, 0, \ldots, 0)$ with $1$ at the $i$-th position. Moreover, for any nonzero $V \in \CMod$, one has $\gd(DV) = \gd(V) - 1$.
\end{lemma}

\begin{proof}
Firstly, note that one has
\begin{equation*}
\Sigma_i M(\bfn) =
\begin{cases}
M(\bfn), & n_i = 0;\\
M(\bfn - \mathbf{1}_i), & n_i \geqslant 1.
\end{cases}
\end{equation*}
This observation implies the isomorphism.

Now we show that $\boldsymbol{\Sigma}$ is an inductive functor. Clearly, it is exact since all $\Sigma_i$ are exact. The naturality $\mathrm{Id}^d \to \boldsymbol{\Sigma}$ is clear. It is also clear that $\boldsymbol{\Sigma} M(\bfn)$ is finitely generated by the isomorphism in the lemma, and the natural map $M(\bfn) \to \boldsymbol{\Sigma} M(\bfn)$ is injective. One needs to check that $\gd(DV) = \gd(V) - 1$ for nonzero $\C$-modules $V$. For a free module $M(\bfn)$, the above isomorphism gives
\begin{equation*}
DM(\bfn) = \bigoplus_{n_i \geqslant 1} \Sigma_i M(\bfn) / M(\bfn).
\end{equation*}
It is not hard to see that $\gd(\Sigma_i M(\bfn) / M(\bfn)) = |\bfn| - 1$ since $\Sigma_i M(\bfn) / M(\bfn))$ is nonzero and generated by its value on the object $\bfn - \mathbf{1}_i$.

Now we consider an arbitrary finitely generated $\C$-module. The case $V=0$ is trivial, so suppose that $V\neq 0$. Let $P \to V$ be a surjective homomorphism where $P$ is a finite direct sum of free modules and $\gd(P)=\gd(V)$. Since $D$ is right exact, we have a surjective homomorphism $DP \to DV$, and hence
\begin{equation*}
\gd(DV) \leqslant \gd(DP) = \gd(P)-1 = \gd(V) - 1.
\end{equation*}
To prove that $\gd(DV) \geqslant \gd(V)-1$, the case $\gd(V)\leqslant 0$ is trivial, so suppose that $\gd(V)>0$. It suffices to prove that for any non-negative integer $n < \gd(V)$, one has $\gd(DV) \geqslant n$. To this end, let $V'$ be the quotient of $V$ by its $\C$-submodule generated by all the elements of $V_{\bfn}$ for every object $\bfn$ of $\C$ with $|\bfn| \leqslant n$. Since $D$ is right exact, we have a surjective homomorphism $DV \to DV'$, so $\gd(DV) \geqslant \gd(DV')$. We shall show that $\gd(DV') \geqslant n$.

If $\bfn$ is any object of $\C$ such that $|\bfn| < n$, then $(\boldsymbol{\Sigma} V')_{\bfn} = 0$ and so $(DV')_{\bfn} =0$. Therefore, we only need to see that $DV' \neq 0$. Since $n <\gd(V)$, we know that $V' \neq 0$. Let $\mathbf{m}$ be a minimal object of $\C$ such that $V'_{\mathbf{m}} \neq 0$. Then $|\mathbf{m}| > n \geqslant 0$. This implies that there is an $i \in [d]$ such that $V'_{\mathbf{m} - \mathbf{1}_i} = 0$ and $(\boldsymbol{\Sigma} V')_{\mathbf{m} - \mathbf{1}_i} \neq 0$, so $(DV')_{{\mathbf{m} - \mathbf{1}_i}} \neq 0$.

\end{proof}

\subsection{The category $\FI^d$}

The category $\FI^d$ is the product of $d$ copies of the category $\FI$; it was introduced by Gadish in \cite{Gad}. Its objects are $d$-tuples of finite sets $\bfn = ([n_1], [n_2], \ldots, [n_d])$, and its morphisms are $d$-tuples of injections $(f_1, f_2, \ldots, f_d)$. Using the rank function $\bfn \mapsto n_1 + n_2 + \ldots + n_d$, one can define the homological degrees of $\FI^d$-modules. The fact that any finitely generated $\FI^d$-module over a commutative Noetherian ring is Noetherian is implied by \cite[Proposition 2.2.4]{SS}, and recently Yu and the second author gave another proof of this fact in \cite[Section 3]{LY2}. Moreover, the category of $\FI^d$-modules has $d$ distinct shift functors $\Sigma_i$, $i \in [d]$. Define the functor $\boldsymbol{\Sigma}$ from the category of $\FI^d$-modules to itself by $\boldsymbol{\Sigma} V = \Sigma_1 V\oplus \cdots \oplus \Sigma_d V$ for every $\FI^d$-module $V$. One has (\cite[Lemma 2.3]{LY2}):
\begin{equation*}
\boldsymbol{\Sigma} M(\bfn) \cong M(\bfn)^{\oplus d} \oplus \left( \bigoplus_{n_i \geqslant 1} M(\bfn - \mathbf{1}_i)^{\oplus n_i} \right).
\end{equation*}
Then $\boldsymbol{\Sigma}$ is an inductive functor for $\FI^d$-modules. Moreover, it satisfies all requirements in Theorems \ref{main theorem I} and \ref{main theorem II}. For details, see \cite[Subsection 2.2]{LY2}.

\subsection{The category $\OI^d$}

The category $\OI^d$ is the product of $d$ copies of the category $\OI$, a subcategory of $\FI$. Its objects are $d$-tuples of finite sets $\bfn = ([n_1], [n_2], \ldots, [n_d])$, and its morphisms are $d$-tuples of non-decreasing injections $(f_1, f_2, \ldots, f_d)$. The fact that any finitely generated $\OI^d$-module over a commutative Noetherian ring is Noetherian is implied by \cite[Proposition 2.2.4]{SS}, and was also explained by Yu and the second author in \cite[Subsection 1.8]{LY2}. Similar to $\FI^d$, the category of $\OI^d$-modules has $d$ distinct shift functors $\Sigma_i$, $i \in [d]$, and we define the functor $\boldsymbol{\Sigma}$ by $\boldsymbol{\Sigma}V = \Sigma_1 V\oplus \cdots\oplus \Sigma_d V$ for every $\OI^d$-module $V$. One can check that:
\begin{equation*}
\boldsymbol{\Sigma} M(\bfn) \cong M(\bfn)^{\oplus d} \oplus \left( \bigoplus_{n_i \geqslant 1} M(\bfn - \mathbf{1}_i) \right).
\end{equation*}
The reader can check (using the same argument for $\FI^d$) that $\boldsymbol{\Sigma}$ is an inductive functor for $\OI^d$-modules and satisfies all requirements in Theorems \ref{main theorem I} and \ref{main theorem II}.

\subsection{The category $\FI_d$}

The category $\FI_d$ was introduced by Sam and Snowden in \cite{SS}. Let us recall its definition. By a \emph{$d$-coloring} on a set $Z$, we mean a map from $Z$ to $[d]$. The category $\FI_d$ has objects $[n]$, $n \in \N$. The morphisms of $\FI_d$ from $[m]$ to $[n]$ are the pairs $(f,\delta)$ where $f: [m] \to [n]$ is any injection and $\delta$ is any $d$-coloring on $[n] \setminus f([m])$. The composition of a morphism $(f, \delta): [l] \to [m]$ with a morphism $(g, \varepsilon): [m] \to [n]$ is defined to be the morphism $(h, \zeta): [l] \to [n]$ where $h=g\circ f$ and
\begin{equation*}
\zeta(z) = \left\{ \begin{array}{ll}
\delta(y) & \mbox{ if } z=g(y), \\
\varepsilon(z) & \mbox{ else. }
\end{array}\right.
\end{equation*}

Sam and Snowden \cite{SS} proved that any finitely generated $\FI_d$-module over a commutative Noetherian ring is Noetherian. Following along the same lines as \cite{CE} and \cite{CEFN} for $\FI$-modules, we first define the shift functor for $\FI_d$-modules and study its basic properties. Define a functor $\iota : \C \to \C$ by $[n] \mapsto [n+1]$, sending $i \to i+ 1$ for $i \in [n]$. If $(f,\delta): [m] \to [n]$ is a morphism of $\C$, then $\iota(f, \delta) : [m+1] \to [n+1]$ is the morphism $(\tilde{f}, \delta)$ such that $\tilde{f} (1) = 1$ and $\tilde{f}(i) = f(i-1) + 1$ for $i \geqslant 2$. The \emph{shift functor} $\Sigma : \CMod \to \CMod$ is defined by
\begin{equation*}
\Sigma V = V \circ \iota \quad \mbox{ for every }V \in \CMod.
\end{equation*}

The following lemma implies that $\Sigma$ is an inductive functor for $\FI_d$-modules satisfying all requirements in Theorems \ref{main theorem I} and \ref{main theorem II}.

\begin{lemma} \label{lemma for FI_d}
Let $\Sigma$ be defined as above. We have:
\begin{enumerate}
\item There is a natural $\FI_d$-module transformation $\theta: \mathrm{Id}^{\oplus d} \to \Sigma$;
\item For any object $[n]$ of $\FI_d$, there is an isomorphism
\begin{equation*}
 M([n])^{\oplus d}  \oplus  M([n-1])^{\oplus n} \longrightarrow \Sigma M([n])
\end{equation*}
whose restriction to $M([n])^{\oplus d}$ is the natural map
\[ \theta : M([n])^{\oplus d} \to \Sigma M([n]).\]
\item Let $S\subseteq [d]$. If $P$ is a projective $\FI_d$-module, then $D_S P$ is a projective $\FI_D$-module. If $V$ is a $\FI_d$-submodule of $P$, then $\Theta_S : V^S \to \Sigma V$ is injective.
\item If $V \neq 0$, then $\gd(DV) = \gd(V) - 1$.
\end{enumerate}
\end{lemma}

\begin{proof}
(1): For any finite set $[n]$ and $\ell \in [d]$, we denote by
\begin{equation*}
\mathrm{incl}_n : [n] \hookrightarrow [n+1] \quad\mbox{ and } \quad \mathrm{col}_\ell: \{ 1 \} \to [d]
\end{equation*}
the inclusion map and the $d$-coloring $\mathrm{col}_\ell(1)=\ell$, respectively. The pair $(\mathrm{incl}_n, \mathrm{col}_\ell)$ is a morphism of $\FI_d$ from $[n]$ to $\iota([n]) = [n+1]$. For each $\ell\in [d]$, there is a natural morphism of functors $\eta_\ell : \mathrm{Id}_{\FI_d} \to \iota$ whose component at each object $[n]$ of $\C$ is the morphism
\begin{equation*}
(\mathrm{incl}_n, \mathrm{col}_\ell) : [n] \longrightarrow \iota([n]) = [n+1].
\end{equation*}
For any $\FI_d$-module $V$, the horizontal composition of morphisms of functors $\mathrm{Id}_V \circ \eta_{\ell}$ gives a natural $\FI_d$-module homomorphism
\begin{equation*}
\theta_{\ell} : V \longrightarrow \Sigma V.
\end{equation*}
Let $\theta = \theta_1 + \ldots + \theta_d$, which is the wanted natural $\FI_d$-module transformation.

(2): This follows easily from the following observation: for each morphism
\begin{equation*}
(f,\delta): [n] \to [r + 1],
\end{equation*}
either $1 \notin f([n])$ in which case $f([n]) \subset [r+1] \setminus \{1\}$ and there is a unique $\ell \in [d]$ such that $\delta(1) = \ell$, or there is a unique $i \in [n]$ such that $f(i)= 1$ in which case $f([n] \setminus \{i\}) \subset [r+1] \setminus \{1\}$.

(3): Any projective $\C$-module is a direct summand of a $\FI_d$-module of the form $\bigoplus_{i \in I} M([n_i])$, so it suffices to prove the assertions when $P$ is $M([n])$, in which case the assertions are immediate from (2).

(4): This can be proved using the same argument as in the proof of Lemma \ref{lemma for N^d}.
\end{proof}

\subsection{The category $\OI_d$}

The category $\OI_d$ introduced by Sam and Snowden \cite{SS} is defined as follows. The objects of $\OI_d$ are the totally ordered finite sets. The morphisms are the pairs $(f,\delta)$ where $f$ is any order-preserving injection and $\delta$ is any $d$-coloring on the complement of the image of $f$. Sam and Snowden \cite[Theorem 7.1.2]{SS} proved that any finitely generated $\mathrm{OI}_d$-module over a commutative Noetherian ring is Noetherian. There is a shift functor $\Sigma$ for $\OI_d$-modules (see \cite[Example 5.8]{GL1}) and the reader can check that all conclusions in Lemma \ref{lemma for FI_d} hold except that the isomorphism in (2) of Lemma \ref{lemma for FI_d} becomes:
\begin{equation*}
\Sigma M([n]) \cong M([n])^{\oplus d} \oplus M([n-1]).
\end{equation*}
Therefore, $\Sigma$ defined as above is an inductive functor for $\OI_d$-modules satisfying all requirements in Theorems \ref{main theorem I} and \ref{main theorem II}.

\subsection{Relative projective modules}

Let $\C$ be one of the following categories: $\CN^d$, $\FI^d$, $\OI^d$, $\FI_d$ and $\OI_d$; and let $V$ be a finitely generated $\C$-module over a commutative Noetherian ring. We have constructed an inductive functor for $\CMod$ satisfying all conditions specified in Theorems \ref{main theorem I} and \ref{main theorem II}. Therefore, by Theorem \ref{main theorem II}, we know that $\reg(V)$ is a finite number, establishing Theorem \ref{main theorem III}, though we have not obtained an explicit upper bound.

In this subsection we deduce more specific information about homology groups of $V$. In particular, we will classify those $\C$-modules whose higher homology groups vanish: they are \emph{relative projective modules}, or \emph{$\sharp$-filtered modules}, a notion introduced by Nagpal for $\FI$-modules in \cite{Nag}.

Let $x$ be a fixed object of $\C$. Then the group algebra $k\C(x, x)$ \footnote{When $\C$ is one of the above five categories, $\C(x, x)$ is a group. Categories satisfying this condition are called \emph{EI} categories, where EI means that endomorphisms are isomorphisms.} is subalgebra of the category algebra $k\C$. Given a $k\C(x, x)$-module $T$, since $M(x)$ is a $(k\C, k\C(x,x))$-bimodule, one can define an \emph{induced module} $M(x) \otimes_{k\C(x, x)} T$. A finitely generated $\C$-module is called a \emph{relative projective module} if $V$ has a filtration whose filtration factors are induced modules. Note that by the finitely generated property of $V$, this filtration must have finite length.

Note that the category $\C$ has a special combinatorial property; that is, for each pair of objects $x$ and $y$ with $x \leqslant y$, the group $\C(x, x)$ acts transitively on the morphism set $\C(x, y)$ (provided that it is nonempty) from the right side. Therefore,
\begin{equation*}
M(x) = \bigoplus_{y \geqslant x} k\C(x, y)
\end{equation*}
is a free right $k\C(x, x)$-module. Therefore, the functor $M(x) \otimes_{k\C(x, x)} -$ is exact. Moreover, it is easy to verify the following fact: for each $k\C(x, x)$-module $T$, one has
\begin{equation*}
H_0 (M(x) \otimes_{k\C(x,x)} T) \cong T.
\end{equation*}

The above fact actually implies the following characterization of relative projective modules.

\begin{proposition} \label{characterizations of relative projective modules}
The following statements are equivalent:
\begin{enumerate}
\item $V$ is relative projective;
\item $H_i(V) = 0$ for all $i \geqslant 1$;
\item $H_1(V) = 0$.
\end{enumerate}
\end{proposition}

\begin{proof}
$(1) \Rightarrow (2)$: Suppose that $V$ is relative projective. Then it has a filtration of finite length whose filtration factors are induced modules. By a standard homological argument, it is enough to show that the higher homology groups of induced modules all vanish; that is, without loss of generality one can suppose that $V = M(x) \otimes_{k\C(x, x)} T$ for a certain object $x$ and a finitely generated $k\C(x, x)$-module $T$.

Take a short exact sequence $0 \to T' \to P \to T \to 0$ of $k\C(x, x)$-modules such that $P$ is projective. Applying the functor $H_0 \circ (M(x) \otimes_{k\C(x, x)} -)$ we recover this short exact sequence. This implies that $H_1(V) = 0$. Replacing $T$ by $T'$ we deduce that $H_2(V) \cong H_1(V') = 0$ where $V' = M(x) \otimes_{k\C(x, x)} T'$. Recursively, one can show that $H_i(V) = 0$ for all $i \geqslant 1$.

$(2) \Rightarrow (3)$: Trivial.

$(3) \Rightarrow (1)$: We use induction on the size of the following set
\begin{equation*}
\Lambda = \{x \in \Ob \C \mid (H_0(V))_x \neq 0 \}.
\end{equation*}
Clearly, we can assume that $\Lambda$ is nonempty since otherwise $V = 0$. Choose a maximal object $x$ in $S$. It is clear that $|x| = \gd(V)$. Let $V'$ be the submodule of $V$ generated by
\begin{equation*}
\bigoplus_{\lambda \in \Lambda \setminus \{ x \}} V_{\lambda}.
\end{equation*}
We obtain a short exact sequence $0 \to V' \to V \to V'' \to 0$. Applying the homology functor we deduce that $H_1(V'')$ is isomorphic to a submodule of $H_0(V')$, and hence $\hd_1(V'') \leqslant \gd(V')$.

Note that $V''$ is a nonzero $\C$-module generated by $V_x''$. Therefore, there is a short exact sequence
\begin{equation*}
0 \to K \to M(x) \otimes _{k\C(x, x)} V_x'' \to V'' \to 0.
\end{equation*}
By considering the long exact sequence of homology groups associated to this short exact sequence, we deduce that
\begin{equation*}
\gd(K) \leqslant \max \{\gd(M(x) \otimes _{k\C(x, x)} V_x''), \, \hd_1(V'') \} \leqslant \max \{|x|, \, \gd(V') \} = |x| = \gd(V).
\end{equation*}
However, it is obvious that if there is an object $y$ such that $K_y \neq 0$, then $y > x$, and hence $|y| > |x|$. In particular, $\gd(K) > |x|$ in this case. Therefore, the only possibility is that $K = 0$. Consequently, $V''$ is an induced module.

Return to the short exact sequence $0 \to V' \to V \to V'' \to 0$. We have shown that $V''$ is an induced module. To prove that $V$ is relative projective, it suffices to verity that $V'$ is relative projective as well. By the induction hypothesis (since $(H_0(V'))_y \neq 0$ if and only if $y \in S \setminus \{x\}$), we only need to show that $H_1(V') = 0$. But this is clear from the long exact sequence of homology groups associated to this short exact sequence, since $H_1(V) = 0$, and $H_i(V'') = 0$ for $i \geqslant 1$ as $V''$ is an induced module.
\end{proof}

\subsection{Minimal linear relative projective resolutions for truncated modules}

As before, let $\C$ be one of the following categories: $\CN^d$, $\FI^d$, $\OI^d$, $\FI_d$ and $\OI_d$; and let $V$ be a finitely generated $\C$-module over a commutative Noetherian ring. We make the following definition.

\begin{definition} \normalfont
The degree of $V$ is defined to be
\begin{equation*}
\deg{V} = \sup \{ |x| \mid V_x \neq 0\}.
\end{equation*}
By convention, we set the degree of the zero module to be $-1$. The $\C$-module $V$ is called a \emph{finitely supported module} if $\deg(V)$ is finite.
\end{definition}

Clearly, $\hd_i(V) = \deg (H_i(V))$. Furthermore, if $\deg(V)$ is finite, then there are only finitely many objects $x$ such that $V_x \neq 0$ since $\C$ is locally bounded.

For each $n \in \N$, we can define a \emph{truncation functor} $\tau_n$ in the category of finitely generated $\C$-modules by the following rule: $\tau_n V$ is a submodule of $V$ satisfying
\begin{equation*}
(\tau_n V)_x =
\begin{cases}
V_x, & \text{if } |x| \geqslant n;\\
0, & \text{otherwise}.
\end{cases}
\end{equation*}
There is a short exact sequence
\begin{equation*}
0 \to \tau_n V \to V \to V/\tau_n V \to 0
\end{equation*}
where $V/\tau_n V$ is a finitely supported module.

Slightly modifying the proof of \cite[Theorem 4.8]{L1}, one can obtain an upper bound for the regularity of finitely supported modules.

\begin{lemma} \label{reg of finitely supported modules}
Let $V$ be a finitely generated $\C$-module. Then $\reg(V) \leqslant \deg(V)$.
\end{lemma}

\begin{proof}
The conclusion holds trivially for the case that $\deg(V) = \infty$. Otherwise, $V$ is a finitely supported module. In the latter case, one can apply the inductive functor constructed previously for $\C$ and easily observes that $\deg(FV) = \deg(V) - 1$ whenever $V$ is nonzero. Therefore, we can use an induction on $\deg(V)$. The conclusion holds for $V = 0$ trivially. For a nonzero $V$, one has
\begin{equation*}
\reg(V) \leqslant \reg(FV) + 1 \leqslant \deg(FV) + 1 = \deg(V) -1 + 1 = \deg(V),
\end{equation*}
where the first inequality follows from Corollary \ref{a corollary}, the second inequality follows from the induction hypothesis since $\deg(FV) = \deg(V) -1$. The conclusion follows from induction.
\end{proof}

Now we restate and prove Theorem \ref{main theorem IV}.

\begin{theorem}
Let $k$ be a commutative Noetherian ring, and let $\C$ be one of the following categories:
\begin{equation*}
\FI^d, \, \OI^d, \, \FI_d, \, \OI_d.
\end{equation*}
Let $V$ be a finitely generated $\C$-module. Then for $n \geqslant \reg(V)$, the truncated module $\tau_n V$ has a resolution $F^{\bullet} \to \tau_n V \to 0$ satisfying the following conditions: for $i \geqslant 0$,
\begin{enumerate}
\item $F^i$ is a relative projective module;
\item $H_0(F^i) \cong H_i(\tau_n V)$;
\item if $F^i \neq 0$, then $F^i$ is generated by its values on objects $x$ with $|x| = n + i$.
\end{enumerate}
In particular, if $H_i(\tau_n V) \neq 0$, then $\hd_i(\tau_n V) = n+i$.
\end{theorem}

\begin{proof}
Firstly, we show that $\hd_i(\tau_n V) \leqslant n + i$. By the long exact sequence
\begin{equation*}
\ldots \to H_{i+1} (V/\tau_nV) \to H_i(\tau_n V) \to H_i(V) \to \ldots,
\end{equation*}
induced by the short exact sequence $0 \to \tau_n V \to V \to V/\tau_n V \to 0$ we have
\begin{equation*}
\hd_i(\tau_n V) \leqslant \max \{\hd_i(V), \, \hd_{i+1} (V/\tau_n V) \}.
\end{equation*}
Clearly
\begin{equation*}
\hd_i(V) \leqslant \reg(V) + i \leqslant n + i.
\end{equation*}
By Lemma \ref{reg of finitely supported modules}, we have
\begin{equation*}
\hd_{i+1}(V/\tau_n V) \leqslant (n - 1) + i + 1 = n + i.
\end{equation*}
Combining these inequalities, we conclude that $\hd_i(\tau_n V) \leqslant n+i$.

Now we construct the linear relative projective resolution. If $\tau_n V$ is a relative projective module, then
\begin{equation*}
0 \to \tau_nV \to \tau_n V \to 0
\end{equation*}
is the wanted relative projective resolution. Otherwise, note that $n \geqslant \reg(V) \geqslant \gd(V)$. Consequently, $\tau_n V$ is generated by the values $(\tau_n V)_x = V_x$ on objects $x$ with $|x| = n$. In other words,
\begin{equation*}
\tau_n V = \sum _{|x| = n} k\C \cdot V_x.
\end{equation*}
Therefore, we have a surjection
\begin{equation*}
F^0 = \bigoplus _{|x| = n} M(x) \otimes_{k\C(x, x)} V_x \to \tau_n V \to 0,
\end{equation*}
where the existence of such a surjection relies on the condition that for two distinct objects $x$ and $y$ with $|x| = |y|$, there is no morphisms between $x$ and $y$. Note that $H_0(F^0) \cong H_0(\tau_n V)$. Let $K^1$ be the kernel of this surjection which is nonzero, and we have
\begin{equation*}
\gd(K^1) \leqslant \max \{ \gd(F^0), \, \hd_1(\tau_n V) \} \leqslant n+1
\end{equation*}
since $\gd(F^0) = n$ and $\hd_1(\tau_n V) \leqslant n + 1$ by the conclusion proved in the previous paragraph. Thus $K^1$ is generated by the values $K^1_x$ with $|x| = n+1$ since $K^1_y = 0$ for $|y| \leqslant n$. Moreover, from the short exact sequence
\begin{equation*}
0 \to K^1 \to F^0 \to \tau_n V \to 0
\end{equation*}
and Proposition \ref{characterizations of relative projective modules} one deduces that $H_i(K^1) \cong H_{i+1} (\tau_n V)$ for $i \geqslant 0$, so $\hd_i(K^1) \leqslant n+i + 1$ for $i \geqslant 0$.

If $K^1$ is a relative projective module, then
\begin{equation*}
0 \to K^1 \to F^0 \to \tau_n V \to 0
\end{equation*}
is the wanted relative projective resolution. Otherwise, replacing $\tau_n V$ by $K^1$ in the previous argument, we can get a short exact sequence
\begin{equation*}
0 \to K^2 \to F^1 = \bigoplus _{|x| = n+1} M(x) \otimes_{k\C(x, x)} K^1_x \to K^1 \to 0
\end{equation*}
with $H_0(F^1) \cong H_0(K^1) \cong H_1(\tau_n V)$ and show that $K^2$ is nonzero and is generated by its values on objects $x$ with $|x| = n+2$. Recursively, we can construct the desired linear relative projective resolution.
\end{proof}

Two immediate corollaries are:

\begin{corollary}
Let $V$ be a finitely generated $\FI$-module which is not a relative projective module. Then for $n \geqslant \gd(V) + \hd_1(V) - 1$, $H_i(\tau_n V)$ is generated in degree $n+i$.
\end{corollary}

\begin{proof}
By \cite[Theorem A]{CE}, $\reg(V) \leqslant \gd(V) + \hd_1(V) - 1$. It suffices to show that $H_i(\tau_n V) \neq 0$. If this happens, then from \cite[Theorem 1.3]{LY1} we deduce that $V$ is a relative projective module, contradicting the given condition.
\end{proof}

\begin{corollary}
Let $V$ be a finitely generated $\C$-module over a field of characteristic 0. Then for every $n \geqslant \reg(V)$, $\tau_nV$ is a Koszul module.
\end{corollary}

\begin{proof}
When $k$ is a field of characteristic 0, the relative projective resolution we constructed is a minimal linear projective resolution of $\tau_n V$.
\end{proof}

\end{document}